\documentclass[a4paper,twoside]{article}

\usepackage{color}

\usepackage{amsmath,amsthm,amsfonts,latexsym,amscd,amssymb,enumerate}
\usepackage[backref]{hyperref}

\swapnumbers
\theoremstyle{plain}
\newtheorem{theorem}{Theorem}[section]
\newtheorem{lemma}[theorem]{Lemma}
\newtheorem{corollary}[theorem]{Corollary}
\newtheorem{proposition}[theorem]{Proposition}

\theoremstyle{definition}
\newtheorem{definition}[theorem]{Definition}
\newtheorem{example}[theorem]{Example}

\theoremstyle{remark}
\newtheorem{remark}[theorem]{Remark}

\newcommand{\cE}{\mathcal{E}}
\newcommand{\cF}{\mathcal{F}}
\newcommand{\cO}{\mathcal{O}}

\newcommand{\cV}{\mathcal{V}}
\newcommand{\cW}{\mathcal{W}}
\newcommand{\bV}{\mathbf{V}}
\newcommand{\bL}{{\bf L}}
\newcommand{\hA}{\hat{\mathbf{A}}}

\newcommand{\reals}{\mathbb{R}}
\newcommand{\complexs}{\mathbb{C}}

\newcommand{\integers}{\mathbb{Z}}
\newcommand{\rationals}{\mathbb{Q}}

\DeclareMathOperator{\id}{id}
\newcommand{\boundary}[1]{\partial#1}

\newcommand{\abs}[1]{\left\lvert#1\right\rvert} %absolute value

\newcommand{\tensor}{\otimes}

\newcommand{\iso}{\cong}
\newcommand{\disjointunion}{\amalg}

\DeclareMathOperator{\Twist}{\mathfrak{Twist}}
\DeclareMathOperator{\Curv}{Curv}
\DeclareMathOperator{\Fred}{Fred}

\DeclareMathOperator{\im}{im}      %image
\DeclareMathOperator{\End}{End}    %Endomorphisms
\DeclareMathOperator{\Hom}{Hom}    %Homomorphisms

\DeclareMathOperator{\tr}{tr}
\DeclareMathOperator{\pr}{pr}
\DeclareMathOperator{\ind}{ind}

\DeclareMathOperator{\ch}{ch}
\DeclareMathOperator{\spinc}{spin^c}
\DeclareMathOperator{\Rham}{Rham}
\DeclareMathOperator{\fla}{flat}

\newcommand{\forget}[1]{}

{\catcode`@=11\global\let\c@equation=\c@theorem}

\allowdisplaybreaks[2]

\begin{document}
\pagestyle{myheadings}
\markboth{Ulrich Bunke and Thomas Schick}{Differential K-theory. A survey.}

%\date{Last compiled \today; last edited  \heuteIst or later}

\title{Differential K-theory. A survey.}

\author{
 Ulrich Bunke\thanks{\protect\href{ulrich.bunke@mathematik.uni-regensburg.de}{ulrich.bunke@mathematik.uni-regensburg.de}} \\
Johannes-Kepler-Universit{\"a}t Regensburg\\
Deutschland
\and 
Thomas Schick\thanks{
\protect\href{mailto:schick@uni-math.gwdg.de}{e-mail:
  schick@uni-math.gwdg.de}
\protect\\
\protect\href{http://www.uni-math.gwdg.de/schick}{www:~http://www.uni-math.gwdg.de/schick}
\protect\\
Fax: ++49 -551/39 2985
}\\
Georg-August-Universit\"at G{\"o}ttingen\\
Deutschland}
\maketitle

\begin{abstract}
  Generalized differential cohomology theories, in particular differential
  K-theory (often called ``smooth K-theory''), are becoming an important
  tool in differential geometry and in mathematical physics.

  In this survey, we describe the developments of the recent decades in this
  area. In particular, we discuss axiomatic characterizations of
  differential K-theory (and that these uniquely characterize differential
  K-theory). We describe several explicit constructions, 
  based on vector bundles, on families of differential operators, or using
  homotopy theory and classifying spaces. We explain the most important
  properties, in particular about the multiplicative structure and
  push-forward maps and will state versions of the Riemann-Roch theorem and of
  Atiyah-Singer family index theorem for differential K-theory.%  We also discuss
%   extensions of these theories, originally defined for smooth manifolds, to
%   certain singular spaces, in particular orbifolds.

\end{abstract}

\tableofcontents
\section{Introduction}

  The most classical differential cohomology theory is ordinary differential
  cohomology \textcolor{black}{
   with integer coefficients. It has various realizations, 
  e.g.~as
  smooth Deligne cohomology 
 (compare \cite{MR1197353}) or as Cheeger-Simons differential characters
 \cite{MR827262}.} In the last decade, \textcolor{black}{differential extensions of}
  generalized cohomology theories, in particular of K-theory, have been
  studied {intensively}. In part, this is motivated by its
  application in mathematical physics, for the description of fields with
  quantization anomalies in abelian gauge theories, suggested by Freed
  in \cite{MR1919425}, compare also \cite{MR2192936}.

  The basic idea is that a differential cohomology theory should combine
  cohomological information with differential form information. More
  precisely, given a generalized cohomology theory $E$ 
  \textcolor{black}{together with a natural transformation}
  $\ch\colon E(X)\to H(X;N)$,
   to cohomology with {coefficients in a graded real vector
     space $N$}, \textcolor{black}{and using an appropriate setup} one can define the 
  differential refinement $\hat E$ of $E$ as a \emph{homotopy} pullback
  \begin{equation*}
    \begin{CD}
      \hat E(X) @>I>> E(X)\\
      @VV{R}V @VV{\ch}V\\
      \Omega_{d=0}(X;N) @>{\Rham}>> H(X;N).
    \end{CD}
  \end{equation*}
  The natural transformations $I$ (the underlying cohomology class) and $R$
  (the characteristic closed differential form) are essential parts of the
  picture. With slight abuse of notation, we call $R$ the \emph{curvature}
  homomorphism. This is a bit of a misnomer, as in a geometric situation $R$
  will be determined by the honest curvature, but not vice versa. $\hat E$ is
  \emph{not} a generalized cohomology theory and not 
  meant to be one: it contains differential form information and as a
  consequence is not homotopy invariant.

  If $E$ is ordinary integral cohomology, $\ch$ is just induced by the
  inclusion of coefficients $\integers\to\reals$. For K-theory, the situation
  we are mainly \textcolor{black}{discussing} in this article, $\ch$ is the ordinary Chern character. 

  The \emph{flat part} $\hat E_{\fla}(X)$ of $\hat E(X)$ is \textcolor{black}{defined as} the kernel of the
  curvature morphism: 
$$\hat E_{\fla}(X):=\ker\left(R\colon \hat E(X)\to
    \Omega(X;{N})\right)\ .$$
 It turns out that $\hat E_{\fla}(X)$ is a cohomology
    theory, usually just $E\reals/\integers[-1]$, the generalized cohomology with
    $\reals/\integers$-coefficients with a degree shift: $\hat
    E^k_{\fla}(X)=E\reals/\integers^{k-1}(X)$. An original interest in
    differential K-theory (before it even was introduced as such) was its
    role as a 
    geometric model for $K\reals/\integers$. Karoubi in \cite[Section
    7.5]{MR913964} defined $K^{-1}\complexs/\integers$ using
    essentially the flat part of a cycle model for $\hat K^0$, compare also
 Lott \cite[Definition 5, Definition 7]{MR1312690} where also
 $K\reals/\integers^{-1}$ is introduced. Homotopy theory provides a
 universal construction of $E\reals/\integers$ for a generalized cohomology
 theory $E$. However, this is in general hard to combine with geometry.

  K-theory is the home for index theory. The differential K-theory (in
  particular its different cycle models) and also its flat part naturally are
  the home for index problems, taking more of the geometry into
  account. Indeed, in suitable models it is built into the definitions that
  geometric families of Dirac operators, parameterized by $X$, give rise to
  classes in $\hat K^*(X)$, where $*$ is the parity of the dimension of the
  fiber. A submersion $p\colon X\to Z$ with closed fibers with fiberwise
  geometric $\spinc$-structure (the precise meaning of geometry will be
  discussed below) is oriented for differential K-theory and one has an
  associated push-forward 
$$\hat p_!\colon \hat K^*(X)\to \hat K^{*-d}(Y)$$ 
(with
  $d=(\dim(X)-\dim(Z))$). The same data also gives rise to a push-forward in
  Deligne cohomology 
$$\hat p_!\colon \hat H^*(X)\to \hat H^{*-d}(Y)\ .$$
 There is
  a unique lift of the Chern character to a natural transformation
  $$\hat\ch\colon \hat K^*(X)\to \hat H_\rationals^{*+2\integers}(X)\ .$$
 Here, the right hand side is the differential extension of
$H^*(X;\rationals)$ and one of the main results of \cite{bunke-2007} 
\textcolor{black}{
is a refinement of the classical  Riemann-Roch theorem to a 
differential Riemann-Roch theorem which identifies the correction for the
compatibility of $\hat \ch$ with the push-forwards.} In \cite{freed10:_index_theor_in_differ_k_theor}, Bismut
 superconnection techniques are used to define the analytic index of a
 geometric family: it can be understood as a particular representative of the
 differential K-theory class of a geometric family as above, determined by the
 analytic solution of the index problem. Moreover, they develop a geometric
 refinement of the topological index construction of Atiyah-Singer
 \cite{MR0236950} based on geometrically controlled embeddings into Euclidean
 space (which does not require deep spectral analysis) and prove that
 topological and analytic index in differential K-theory coincide.

Finally, we observe that, in suitable special situations, we can easily
construct classes in differential K-theory which turn out not to depend on the
special geometry, but only on the underlying differential-topological
data. Typically, these live in the flat part of differential K-theory and are
certain (generalizations of) secondary index invariants. Examples are
rho-invariants of the Dirac operator twisted with two flat vector bundles (and
family versions hereof), or the $\integers/k\integers$-index of Lott
\cite{MR1312690} for a manifold $W$ whose boundary is identified with the
disjoint union of $k$ copies of a given manifold $M$.

\bigskip
  Similar to smooth Deligne cohomology, there is a counterpart of differential
  K-theory in the holomorphic setting \cite{MR1043268} and there is an
  arithmetic Riemann-Roch for these groups. This, however, will not be
  discussed in this survey.

\subsection{Differential cohomology and physics}
\label{sec:diff_cohom_and_physcis}

 A motivation for the introduction of differential K-theory comes from quantum
 physics. The fields \textcolor{black}{of abelian gauge theories are
 described by objects which carry the local field strength  information of a 
 closed differential form (assuming that there are no sources). Dirac
 quantization, however, requires that their de 
 Rham classes lie in an integral lattice in de Rham 
 cohomology. }For Maxwell theory \textcolor{black}{the field strength} simply is a $2$-form which is the
 curvature of a complex line bundle and therefore lies in the image of
 ordinary integral cohomology. For Ramond-Ramond fields in type II string
 theories it is a differential form of higher degree which lies in the image
 under the Chern character of K-theory, as suggested by
 \cite{Moore-Witten,freed_hopkins}. Indeed, Witten suggests that D-brane
 charges in the low energy limit of type IIA/B superstring theory are
 classified by K-theory. In this case,
 even if the field strength differential form is
 zero, the fields or D-brane charges can contain some \emph{global}
 information, corresponding to torsion in K-theory.

 It is suggested by
 Freed \cite{MR1919425} that these Ramond-Ramond fields are described by
 classes in 
\emph{ differential} K-theory (or other generalized differential cohomology
theories,
 depending on the particular physical model). Given a space-time background
 $X$ and a field represented by a
 class $F\in \hat E^*(X)$, this field contains the differential form
 information $R(F)$ (as expected for an abelian gauge
 field). The field equations (generalizing Maxwell's equations) require
 that $dR(F)=0$ if there are no {sources} (which we assume here). However,
 there is a quantization condition: the de Rham class represented by $R(F)$ is
 not arbitrary, but lies in an integral lattice, namely in $\im(ch)$. Indeed,
 $F$ also contains the integral (and possibly torsion) information of the
 class $I(F)\in E^*(X)$. Finally, even $I(F)$ and $R(F)$ together don't
 determine $F$ entirely, there is extra information,
 corresponding to a physically significant potential or holonomy. More
 precisely, $\hat E^*(X)$ is the configuration space with a gauge group
 action. 
 Details of such a gauge field theory are studied e.g.~in \cite{MR2536859},
 where 
 it is shown that the free part of $E^*(X)$ is an obstruction to a global
 gauge fixing. Nonetheless, \cite{MR2536859} proposes a partition function
 and among others computes the vacuum expectation value. 

All discussed so far describes the situation without any background
field or flux. However, such background fields are an important ingredient of
the 
theory. Depending on the chosen model and the precise situation, a
background field can be defined in many different ways. In the classical
situation where fields are just given by differential forms, a background
field is a closed $3$-form $\Omega$. It creates an extra term in the field
equations. Correspondingly, the relevant charges are even or odd forms
(depending on the type of the theory) which are closed for the differential
$d^\Omega$ with $d^\Omega\omega:= d\omega+\omega\wedge\Omega$. And they are
classified up to equivalence by the \emph{$\Omega$-twisted de Rham cohomology}
\begin{equation*}
H_{dR}^{*+\Omega}(X):=\ker(d^\Omega)/\im(d^\Omega).
\end{equation*}

When looking at charges in the presence of a background B-field (producing an
H-flux) which are
classified topologically by K-theory, we need to work with
\emph{twisted K-theory}, compare in particular
\cite{MR1674715,MR1756434}. We will give a short introduction to twisted
K-theory in Section 
\ref{sec:twisted_ordinary_K_theory}. The role of twisted K-theory is discussed
a lot in the case of T-duality. T-duality predicts an isomorphism of string
theories on \textcolor{black}{different} background manifolds {which are T-dual to each other}, and in particular an isomorphism of
the K-theory groups which classify the D-brane charges.

It turns out, however, that the topology of one of the partners in duality
dictates a background B-field on the other, and the required isomorphism can
only hold in twisted K-theory,
compare e.g.~\cite{MR2080959,MR2130624}. In those papers, mainly the
\emph{topological} classification of D-brane charges is considered. A new
picture now arises when one wants to move to T-duality for Ramond-Ramond
fields described by \emph{differential K-theory} as explained above. One has
to construct and study \emph{twisted} differential K-theory. A first step
toward this is carried out in \cite{MR2518992}. Now, physicists
try to understand T-duality at the level of Ramond-Ramond fields, compare
e.g.~\cite{MR2628876} where the ideas are discussed explicitly without
mathematical rigor. With mathematical rigor, the T-duality isomorphism in
(twisted) differential K-theory has
been worked out by Kahle and Valentino in \cite{kahle-2009}. We will
describe these results in more detail in Section \ref{sec:diff_T_duality}.

\section{Axioms for differential cohomology}
\label{sec:axioms}

A fruitful approach to generalized cohomology theories is based on the
Eilenberg-Steenrod axioms. It turns out that many of the basic properties of
smooth Deligne cohomology and differential K-theory also are captured by a
rather small set of axioms, proposed in \cite[Section
1.2.2]{bunke-2007} (and motivated by \cite{MR1919425}). Therefore, we want to
base our treatment of differential 
K-theory on those axioms, as well.

The starting point is a generalized cohomology theory $E$, together with a
natural transformation $\ch\colon E(X)\to H(X;N)$, where $N$ is a graded
coefficient $\reals$-vector space. The two basic examples are
\begin{itemize}\item 
  $E(X)=H(X;\integers)$, ordinary cohomology with integer coefficients, where
  $N=\reals$ and $\ch$ is induced by the inclusion of coefficients
  $\integers\to\reals$. More generally, $\integers$ can be replaced by any
  subring of $\reals$, e.g.~by $\rationals$.
\item $E(X)=K(X)$, K-theory, where $\ch$ is the usual Chern character, and
  $N=\reals[u,u^{-1}]$ with $u$ of degree {$2$}. Multiplication with $u$
  corresponds to Bott periodicity. 
\end{itemize}

\begin{definition}\label{def:axioms}
  A differential extension of the pair $(E,\ch)$ is a functor $X\to\hat E(X)$ from the
  category of compact smooth manifolds (possibly with boundary) to
  $\integers$-graded groups together with natural transformations
  \begin{enumerate}
  \item $R\colon \hat E(X)\to \Omega_{d=0}(E;N)$ (curvature)
  \item $I\colon \hat E(X)\to E(X)$ (underlying cohomology class)
  \item $a\colon \Omega(X;N)/\im(d)\to \hat E(X)$ (action of forms)\ .
  \end{enumerate}
   Here $\Omega(E;N):=\Omega(E)\tensor_{\reals}N$\footnote{{This definition has to be modified in a generalization to non-compact manifolds!}} denote the smooth
  differential forms with values in $N$, $d\colon
  \Omega(E;N)\to\Omega(E;N)[1]$ the usual de Rham differential and
  $\Omega_{d=0}(E;N)$ the space of closed differential forms ($[1]$ stands for
  degree-shift by $1$).

  The transformations $I,a,R$ are required to satisfy the following axioms:
  \begin{enumerate}
  \item The following diagram commutes
\begin{equation}\label{eq:main_diagr}
\begin{CD}
  \hat E(X) @>I>> E(X)\\
  @VV{R}V    @VV{\ch}V\\
      \Omega_{d=0}(X,N) @>{\Rham}>> H(X;N)\; .
\end{CD} 
\end{equation}
\item\begin{equation}\label{drgl} R\circ a=d\ .
  \end{equation}
\item $a$ is of degree $1$.
  % and satisfies $a(\omega)\cup x=a(\omega\wedge R(x))$ for $x\in \hat h(B)$
  % and $\omega\in\Omega(B,N)/\im(d)$.
\item The following sequence is exact:
  \begin{equation}\label{exax}
    E^{*-1}(X)\xrightarrow{ch} \Omega^{*-1}(X,N)/\im(d)\xrightarrow{a} \hat
    E^*(X)\xrightarrow{I} E^*(X)\to 0\ . 
  \end{equation}
\end{enumerate}
\end{definition}

Alternatively, when dealing with K-theory one can
  and often (e.g.~in \cite{bunke-2007}) does consider the whole theory as
  $\integers/2\integers$-graded with the obvious adjustments. Note that with
  $N=K^*(pt)\tensor\reals=\reals[u,u^{-1}]$ we have natural and canonical
  isomorphism $\Omega^*(X;N)=\oplus_{k\in \integers} \Omega^{*+2k}(X)$ and
  $H^*(X;N)=\oplus_{k\in\integers} H^{*+2k}(X;\reals)$. The associated
  $\integers/2\integers$-graded ordinary cohomology is therefore given by the
  direct sum of even or odd degree forms.

  \begin{corollary}
    If $\hat E$ is a differential extension of $(E,\ch)$, then we have a
    second exact sequence
    \begin{equation}\label{eq:les_R_surj}
  E^{*-1}(X)\xrightarrow{\ch}  H^{*-1}(X;N)  \xrightarrow{a}\hat E^*(X) \xrightarrow{R} \Omega^*_{d=0}(X;N)\times_{\ch} E^*(X)\to 0,
    \end{equation}
    where $\Omega_{d=0}(X;N)\times_{\ch} E(X)=\{(\omega,x)\mid
    \Rham(\omega)=\ch(x)\}$ is the pullback of abelian groups.
  \end{corollary}
  \begin{proof}
    This is a direct consequence of \eqref{exax}, \eqref{eq:main_diagr} and \eqref{drgl}.
  \end{proof}

\begin{definition}\label{def:flat_part}
    Given a differential extension $\hat E$ of a cohomology theory $(E,\ch)$,
    we define the associated \emph{flat functor}
    \begin{equation*}
      \hat E_{flat}(X):=\ker(R\colon \hat E(X)\to \Omega_{d=0}(X;N)).
    \end{equation*}
  \end{definition}

\begin{remark}
  The naturality of $R$ indeed implies that $X\mapsto \hat E_{flat}(X)$ is a
  {contravariant} functor on the category of smooth manifolds. Actually, this
  functor by Corollary \ref{corol:flat_is_homotopy_inv} is always homotopy
  invariant and  {extends to} a cohomology theory {in many examples}, as we 
  will discuss in Section \ref{sec:flat_part}. Typically, there is a natural
  isomorphism $\hat 
  E^*_{flat}(X)\iso E\reals/\integers^{*-1}(X) $, but we still don't
  know whether this is necessarily always the case (compare the discussion in
  \cite[Section 5 and Section 7]{BSuniqueness}). 
  \end{remark}

The most interesting cases are not just group valued cohomology functors, but
multiplicative cohomology theories, for example K-theory and ordinary
cohomology. We therefore want typically a differential extension which carries
a compatible product structure. 

\begin{definition}\label{multdef1}
Assume that $E$ is a multiplicative cohomology theory, {that $N$ is a $\integers$-graded algebra over $\reals$},  {and that $\ch$ is
compatible with the ring structures}. A differential extension $\hat E$ of
$(E,\ch)$ is called multiplicative if $\hat E$ together with the
transformations $R,I,a$ is a differential extension of $(E,\ch)$, and in
addition  
\begin{enumerate}
\item $\hat E$ is a functor to $\integers$-graded rings,
\item $R$ and $I$ are multiplicative, 
\item $a(\omega)\cup x=a(\omega\wedge R(x))$ for all $x\in \hat E(B)$ and
  $\omega\in\Omega(B;N)/\im(d)$.  
\end{enumerate}
\end{definition}
Deligne cohomology is multiplicative \cite[Chapter 1]{MR1197353},
\cite[Section 6]{MR2179587}, \cite[Section
4]{bunke10:_geomet_descr_of_smoot_cohom}. In this paper, we will
 consider multiplicative extensions $\hat K$ of
$K$-theory.

\subsection{Variations of the axiomatic approach}\label{sec:variations_of_axioms}

Our list of axioms for differential cohomology theories seems
particularly natural: it allows for efficient constructions and to derive the
conclusions we are interested in. However, for the differential refinements of
integral cohomology, a slightly
different system of axioms has been proposed in \cite{MR2365651}. The
main point there is that the requirement of a given natural isomorphism $\hat
E^*_{flat}(X)\to E\reals/\integers^{*-1}(X)$ between the flat part of
Definition \ref{def:flat_part} and $E$ with coefficients in
$\reals/\integers$. It turns out that, for differential extensions of ordinary
cohomology, both sets of axioms imply that there is a unique natural
isomorphism to Deligne cohomology (compare \cite{MR2365651} and
\cite[Section 7]{BSuniqueness}). In particular, for ordinary cohomology they are
equivalent. The corresponding result holds in general under extra
assumptions, which are satisfied for K-theory, compare Section
\ref{sec:flat_part}.

\subsection{Homotopy formula}

A simple, but important consequence of the axioms is the homotopy formula. If
one differential cohomology class can be deformed to another, this formula
allows to compute the difference of the two classes entirely in terms of
differential form information. In a typical application, one will deform an
unknown class to one which is better understood and that way get one's hands
on the complicated class one started with.

\begin{theorem}[Homotopy formula]\label{theo:homotopy_formula}
  Let $\hat E$ be a differential extension of $(E,\ch)$. If $x\in \hat
  E([0,1]\times B)$ and $i_k\colon B\to [0,1]\times B; b\mapsto (k,b)$ are the 
  inclusions then
  \begin{equation*}
    i_1^*x- i_0^*x = a\left(\int_{[0,1]\times B/B} R(x)  \right),
  \end{equation*}
  where $\int_{[0,1]\times B/B}$ denotes integration of differential forms over
    the fiber of the projection $p\colon [0,1]\times B/B$ with the canonical
    orientation of the fiber $[0,1]$.  
\end{theorem}
\begin{proof}
   Note that, if $x=p^*y$ for some $y\in \hat E(B)$ then by naturality the
   left hand side of the equation is zero. Moreover, in this case
   $R(x)=p^*R(x)$ so that by the properties of integration over the fiber the
   right hand side vanishes, as well.

   In general, observe that $p$ is a homotopy equivalence, so that we always
   find $\bar y\in E({B})$ with $I(x)=p^*\bar y$. Using surjectivity of $I$, we
   find $y\in \hat E({B} )$ with $I(y)=\bar y$, and then $I(p^*y-x)=p^*\bar y -
   \bar y =0$. Since the sequence \eqref{exax} is exact, there is $\omega\in
   \Omega(B;N)$ with $a(\omega)=x-p^*y$. Stokes' theorem applied to $\omega$
   yields
   \begin{equation*}
     i^*_1\omega -i^*_0\omega = \int_{[0,1]\times B/B} d\omega.
   \end{equation*}
   On the other hand, because of \eqref{drgl}, $d\omega =
   R(a(\omega))$. Substituting, we get
   \begin{equation*}
     \int_{[0,1]\times B/B} d\omega =\int_{[0,1]\times B/B} R(a(\omega))
     =\int_{[0,1]\times B/B} R(x-p^*y) = \int_{[0,1]\times B/B} R(x)
   \end{equation*}
   and (using again vanishing of our expressions for $p^*y$)
   \begin{equation*}
     i^*_1x-i^*_0 = i^*_1a(\omega)=i^*_0a(\omega) = a(i^*_1\omega
     -i^*_0\omega) = a\left(\int_{[0,1]\times B/B} R(x)\right). 
   \end{equation*}

\end{proof}

\begin{corollary}\label{corol:flat_is_homotopy_inv}
  Given a differential extension $\hat E$ of $(E,\ch)$, the associated flat
  functor $\hat E_{flat}$ of Definition \ref{def:flat_part} is homotopy
  invariant. 
\end{corollary}
\begin{proof}
  Let $H\colon [0,1]\times X\to Y$ be a homotopy between $f=H_0$ and $g=H_1$. 
  We have to show that $f^*=g^*\colon \hat E_{flat}(Y)\to \hat
  E_{flat}(X)$. By functoriality, it suffices to show that $i_0^*=i_1^*\colon
  \hat 
  E_{flat}([0,1]\times X)\to \hat E_{flat}(X)$, as $f^*=i_0^*\circ H^*$ and
  $g^*=i_1^*\circ H^*$. This, however, follows immediately from Theorem
  \ref{theo:homotopy_formula} once $R(x)=0$.
\end{proof}

We have seen above that $E_{flat}^*$ is a homotopy invariant functor. Ideally,
it should {extend to} a generalized cohomology theory (compare the discussion of
Section \ref{sec:variations_of_axioms}). For this, we need a bit of extra
structure which corresponds to the suspension isomorphism, and which is
typically easily available. We formulate this in terms of integration over the
fiber for $X\times S^1$, originally defined in \cite[Definition
1.3]{BSuniqueness}. Note that the projection $X\times S^1\to X$ is canonically
oriented  
for an arbitrary cohomology theory because the tangent bundle of $S^1$ is
canonically trivialized. For a push-down in differential cohomology, in
general we expect that one has to choose geometric data of the fibers, which
again we can assume to be canonically given for the fiber $S^1$. Orientations
and push-forward homomorphisms for differential cohomology theories, in
particular for differential K-theory are discussed in Section
\ref{sec:integration}. 
 
\begin{definition}\label{def:integration_over_S1}
  We say that a differential extension $\hat E$ of a cohomology theory
  $(E,\ch)$ has \emph{$S^1$-integration} if there is a natural transformation
  $\int_{X\times S^1/X} \colon \hat E^*(X\times S^1)\to \hat E^{*-1}(X)$ which
  is compatible 
  with the transformations $R,I$ and the ``integration over the fiber''
  $\Omega^*(X\times S^1;{N})\to \Omega^{*-1}(X;{N})$ as well as $E^*(X\times S^1)\to
  E^{*-1}(X)$. In addition, we require that $\int_{X\times S^1/X} p^*x=0$ for
  each $x\in \hat 
  E^*(X)$ and $\int_{X\times S^1/X}(\id\times t)^*x=-x$ for each $x\in \hat
  E^*(X\times S^1)$,
  where $t\colon S^1\to S^1$ is complex conjugation. 
\end{definition}

In \cite[Corollary 4.3]{BSuniqueness} we prove that in many situations, e.g.~for
ordinary cohomology or for K-theory, there is a canonical choice of
integration transformation. 
\begin{theorem} \label{theo:integration_exists}
  If $\hat E$ is a multiplicative differential extension of $(E,\ch)$ and if
  $E^{-1}(pt)$ is a torsion group, then $\hat E$ has a canonical
  $S^1$-integration
  as in Definition \ref{def:integration_over_S1}.
\end{theorem}

\subsection{$\hat E_{flat}^*$ as generalized cohomology theory $E\reals/\integers$}
\label{sec:flat_part}

{Let $E$ be a generalized cohomology theory. In the present section we consider a universal differential extension $\hat E$, i.e.~we take $N:=E(*)\otimes_{\integers}\reals$ and let $\ch:E(X)\to H(X;N)$ be the canonical transformation.}

To $E$ there is an associated
generalized cohomology theory $E\reals/\integers$ ($E$ with coefficients in
$\reals/\integers$). It is {constructed} with the help of stable homotopy
theory: the cohomology theory $E$ is given by a spectrum (in the sense of
stable homotopy theory) $\mathbf{E}$, and $E\reals/\integers$ is given by the
spectrum $\mathbf{E}\smash M(\reals/\integers)$, where $M(\reals/\integers)$
is the Moore spectrum of the abelian group
$\reals/\integers$. $E\reals/\integers$ is constructed in such a way that one
has natural long exact sequences
\begin{equation*}
  \to E^*(X)\to E\reals^*(X)\to E\reals/\integers^*(X)\to
  E^{*+1}(X)\to\dots 
\end{equation*}
Note that for {a finite $CW$-complex $X$} one can alternatively write $E^*(X)\tensor \reals=E\reals^*(X)$ with
$E\reals$ defined by smash product with $M(\reals)$.

In the fundamental paper \cite{MR2192936}, Hopkins and Singer construct a
specific differential extension for any generalized cohomology theory $E$. For
this particular construction, one has by \cite[(4.57)]{MR2192936} a natural
isomorphism 
\begin{equation*}
  E\reals/\integers^{*-1}(X) \to E^*_{flat}(X).
\end{equation*}

However, this is a consequence of the particular model used in
\cite{MR2192936}. It is therefore an interesting 
question to which extend the axioms alone imply that $\hat E_{flat}$ is a
generalized cohomology theory. Here, some extra structure about the suspension
isomorphism seems to be necessary, implemented by the transformation
``integration over $S^1$'' of Definition \ref{def:integration_over_S1}. With a
surprisingly complicated proof one gets \cite[Theorem 7.11]{BSuniqueness}:

\begin{theorem} \label{theo:flat_is_cohom}
If $(\hat E,R,I,a,\int)$ is a differential extension of $E$ with integration
over $S^1$, then  $\hat E^*_{flat}$ 
has natural long exact Mayer-Vietoris sequences. It
is equivalent to the restriction  to compact manifolds of a generalized
cohomology theory represented by a spectrum. Moreover, $a\colon E\reals^*\to
\hat E_{flat}^{*+1}$ and $I\colon \hat E_{flat}^*\to E\reals/\integers^{*}$ are
natural transformations of cohomology theories and one obtains a natural long
exact sequence
for each finite CW-complex $X$
\begin{equation}\label{eq:les_flat}
  \to E\reals^{*-1}(X)\xrightarrow{a} \hat E_{flat}^*(X)\xrightarrow{I}
  E^*(X)\xrightarrow{ch}   E\reals^*(X)\to .
\end{equation}
\end{theorem}
\begin{proof}
  The long exact sequence \eqref{eq:les_flat} is not stated like that in
  \cite{BSuniqueness} and we will need it below, therefore we explain how this
  is achieved. Because of \eqref{drgl}, the restriction of $a$ to
  $E\reals^{*-1}(X)$ (realized as de Rham cohomology, i.e.~as $\ker(d\colon
  \Omega^{*-1}(X;N)/\im(d)\to \Omega^*(X;N))$) hits exactly $\hat
  E_{flat}(X)$. The exactness at $\hat E_{flat}(X)$ therefore is a direct
  consequence of the exactness of \eqref{exax}. \eqref{exax} also implies
  immediately that $\ker(a)=\im(ch)$ in \eqref{eq:les_flat}.  Because of
  \eqref{eq:main_diagr}, $ch\circ I=0$ in \eqref{eq:les_flat}. Finally, as
  $I\colon \hat E^*(X)\to E(X)$ is surjective, any $x\in \ker(ch)$ can be
  written as $I(y)$ with $y\in \hat E^*(X)$ and such that $R(y)\in \im(d)$,
  i.e.~$R(y)=d(\omega)$ for some $\omega\in \Omega^{*-1}(X)/\im(d)$. But then
  $x=I(y-d\omega)$ and $y-d\omega\in E^*_{flat}(X)$, which implies that
  \eqref{eq:les_flat} is also exact at $E^*(X)$. 
\end{proof}

\begin{corollary}\label{corol:MV_for_diff}
  Let $(\hat E,R,I,a,\int)$ be a differential extension of $E$ with
  integration over $S^1$, or more generally assume that $\hat E_{flat}$ has
  natural long exact Mayer-Vietoris sequences.
{
  Assume that $X_{1}, X_{2},X_{0}:=X_{1}\cap X_{2}\subseteq X$ are closed submanifolds of codimension $0$
   with boundary (and corners) such the interiors of $X_{1}$ and $X_{2}$ cover $X$.} Then one has a long exact Mayer-Vietoris sequence
  \begin{equation*}
  \hat E^{n+1}_{flat}(X_0)\xrightarrow{i\circ\delta}  \hat E^n(X)\to \hat E^n(X_1)\oplus \hat
  E^n(X_2)\to \hat E^n(X_0)\xrightarrow{\delta\circ I} E^{n-1}(X)\cdots
  \end{equation*}
  which continues to the right with the Mayer-Vietoris sequence for $E$ and to
  the left with the Mayer-Vietoris sequence for $\hat E$.
\end{corollary}
\begin{proof}
  The proof is an standard diagram chase, using the Mayer-Vietoris sequences
  for $E$ and $E_{flat}$, the short exact sequence
  \begin{equation*}
    0\to \Omega^*(X)\to \Omega^*(X_1)\oplus \Omega^*(X_2)\to \Omega^*(X_0)\to 0,
  \end{equation*}
  {the homotopy formula}, and the exact sequences \eqref{exax}.
\end{proof}

\begin{theorem}\label{theo:flat_is_RZ}
If, in addition to the assumption of Theorem \ref{theo:flat_is_cohom},  $E^k$
is finitely generated for each $k\in\integers$ and the torsion
subgroup 
$E^*_{tors}(pt)=0$ then there is an isomorphism of cohomology theories $\hat
E^*_{flat}\xrightarrow{\iso} E\reals/\integers^{*+1}$. 
\end{theorem}
\begin{proof}
  We claim this statement as \cite[Theorem 7.12]{BSuniqueness}. However, the
  proof given there is not correct, and the assertion of \cite[Theorem
  7.12]{BSuniqueness} unfortunately is slightly stronger
  than the one we can actually prove, namely Theorem \ref{theo:flat_is_RZ}. 

  As  by Theorem \ref{theo:flat_is_cohom} $\hat E_{flat}$ is a generalized
  cohomology theory, it is represented by a spectrum $U$, and the natural
  transformation $a$ by a map of spectra. We extend this to a fiber sequence
  of spectra $F\to E\reals\xrightarrow{a} U$, inducing for each compact
  CW-complex $X$ an associated long exact sequence
  \begin{equation}
    \label{eq:FEU_seq}
     \to F^*(X)\to E\reals^*(X)\xrightarrow{a} \hat E_{flat}^*(X)\to.
  \end{equation}
  Comparison with \eqref{eq:les_flat} implies that the image of $F(X)$ in
  $E\reals(X)$ 
  coincides with the image of $E(X)$. This means by definition that the
  composed map of spectra  $F\to E\reals\to E\reals/\integers$ is a phantom
  map. However, under the assumption that the $E^k$ is finitely generated for
  each 
  $k$, it is shown in \cite[Section 8]{BSuniqueness} that such a phantom map
  is automatically trivial. Using the triangulated structure of the homotopy
  category of spectra, we can choose $\phi$, $\phi_F$ to obtain a map of fiber
  sequences (distinguished 
  triangles)
  \begin{equation*}
    \begin{CD}
      F @>>> E\reals @>a>> U\\
      @VV{\phi_F}V @VV=V @VV{\phi}V\\
      E @>>> E\reals @>>> E\reals/\integers.
    \end{CD}
  \end{equation*}
with associated diagram of exact
  sequences  which because of the knowledge about image and kernel of
  $a$ specializes to 
  \begin{equation}\label{eq:special}
    \begin{CD}
      E^*(X) @>>> E\reals^*(X) @>a>> \hat E^{*+1}_{flat}(X) @>>>
      E^{*+1}_{tors}(X)\to 0\\
      @VV=V @VV=V @VV{\phi}V \\
      E^*(X) @>>> E\reals^*(X) @>>> E\reals/\integers^*(X) @>>>
      E^{*+1}_{tors}(X) \to 0.
    \end{CD}
  \end{equation}
  Note that we do not claim (and don't know) whether the diagram can be
  completed to a commutative diagram by $\id\colon E_{tors}^{*+1}(X)\to
  E_{tors}^{*+1}(X)$. 

  If $E_{tors}^*(pt)=0$, the $5$-lemma implies that $\phi\colon U\to
  E\reals/\integers$ induces an isomorphism on the point and therefore for all
  finite CW-complexes.
\end{proof}

\section{Uniqueness of differential extensions}

\label{sec:Uniqueness}

Given the many different models of differential extensions of K-theory, many
of which we are going to
described in Section \ref{sec:models}, it is reassuring that the resulting
theory is uniquely determined. The corresponding statement for differential
extensions of ordinary cohomology has been established in \cite{MR2365651} by
Simons and Sullivan. For K-theory and many other generalized cohomology
theories we will establish this in the current section.

{
As in Subsection \ref{sec:flat_part} we consider universal differential extensions of $E$. 
}
Given two extensions $\hat E$ and $\hat E'$ of 
$E$ with corresponding natural transformations $a,I$ as in \ref{def:axioms},
we are looking for a natural isomorphism $\Phi\colon \hat E\to \hat E'$
compatible with the natural transformations. Provided such a natural
transformation exists, we ask whether it is unique. We have seen that it often
is natural to have additionally a transformation ``integration over $S^1$'' as
in Definition \ref{def:integration_over_S1}, and we require that $\Phi$ is
compatible with this transformation, as well.

To construct the transformation $\Phi$ in degree $k$ there is the following
basic strategy: 
\begin{itemize}
\item find a classifying space $B$ for $E^k$ with universal element $u\in
  E^k(B)$. This means that for any space $X$ and $x\in E^k(X)$, there is a map
  $f\colon X\to B$ (unique up to homotopy) such that $x=f^*u$.
\item lift this universal element to $\hat u\in \hat E^k(B)$ and show that
  this class is universal for $\hat E^k$, or at least that for a class $\hat
  x\in \hat E^k(X)$ one can find $f\colon X\to B$ such that the difference
  $\hat x-f^*\hat u$ is under good control.
\item Obtain similarly $\hat u'\in (\hat E')^k(B)$. Define the transformation
  by $\Phi(\hat u)=\hat u'$ and extend by naturality.
\item Check that $\Phi$ has all desired properties.
\item Uniqueness of $\Phi$ does follow if the lifts $\hat u,\hat u'$ are
  uniquely determined by $u$ once their curvature is also fixed.
\end{itemize}

There are a couple of obvious difficulties implementing this strategy. The
first is the fact, that the classifying space $B$ almost never has the
homotopy type of a finite dimensional manifold. Therefore, $\hat E(B)$ is not
defined. This is solved by replacing $B$ by a sequence of manifolds
approximating $B$ with a compatible sequence of classes in $\hat E^k$ which
replace $\hat u$. Then, the construction of $\Phi$ as indicated indeed is
possible. However, a priori this has a big flaw. $\Phi$ is not necessarily a
transformation of abelian groups. Because of the compatibilities with $a$, $R$
and $I$ the deviation from being additive is rather restricted and in the end
is a class in $E\reals^{k-1}(B\times B)/\im(ch)$. The
different possible transformation are by naturality and the compatibility
conditions determined by the different lift $\hat u'$ with fixed
curvature. This indeterminacy is given by an element in
$E\reals^{k-1}(B)/\im(ch)$. If $k$ is even and $E$ is rationally even, then so
is $B$ as classifying space of $E^k$. It then follows that
$E\reals^{k-1}(B\times B)/\im(ch)=\{0\}$ and $E\reals^{k-1}(B)/\im(ch)=\{0\}$,
i.e. $\Phi$ automatically is additive and unique.

For $k$ odd, the transformation can then be defined (and is uniquely
determined) by the requirement that it is compatible with $\int_{X\times
  S^1/X}$. This construction has been carried out in detail in
\cite{BSuniqueness} and we arrive at the following theorem.

\begin{theorem}\label{theo:uniqueness}
Assume that  $(\hat E,R,I,a,\int)$ and  $(\hat
E^\prime,R^\prime,I^\prime,a^\prime,\int')$ are two differential extensions
with $S^1$-integration of a 
generalized cohomology theory $E$ which is rationally even,
i.e.~$E^{2k+1}(pt)\tensor \rationals=0$ for all $k\in\integers$. Assume
furthermore that $E^k(pt)$ is a finitely generated abelian group for each
$k\in\integers$. Then there is a unique natural isomorphism between these
differential extensions compatible with the $S^1$-integrations.

If no $S^1$-orientation is given, the natural isomorphism can still be
constructed on the even degree part.

If $\hat E$ and $\hat E'$ are multiplicative, the 
transformation is automatically multiplicative. Note that the assumptions
imply by Theorem 
\ref{theo:integration_exists} that then there is a canonical integration.

If $\hat E$ and $\hat E'$ are defined on all manifold, not only on compact
manifolds (possibly with boundary), then it suffices to require that $E^k(pt)$
is countably generated for each $k\in\integers$, and the same assertions hold.
\end{theorem}
\begin{proof}
  The proof is given in \cite{BSuniqueness} and we don't plan to repeat it
  here. However, there we made the slightly stronger assumption that
  $E^{2k+1}(pt)=0$ if $\hat E$ is only defined on the category of compact
  manifolds. Let us therefore indicate why the stronger result also holds. As
  described in the strategy, the first task is to approximate the classifying
  spaces $B$ for $E^{2k}$ by spaces on which we can evaluate $\hat E^*$. These
  approximations are constructed inductively by attaching handles to obtain
  the correct 
  homotopy groups (introducing new homotopy, but also killing superfluous
  homotopy). To construct a compact manifold, we are only allowed to
  attach finitely many handles. Therefore, we have to know a priori that we
  have to kill only a finitely generated homotopy groups. In \cite{BSuniqueness}
  we assume that $\pi_1(B)=E^{2k-1}(pt)$ is zero, to be allowed to start with
  a simply connected approximation. Then we use that all homotopy groups of a
  simply connected finite CW-space are finitely generated. However, exactly
  the same holds for finite CW-spaces with finite fundamental group, because
  the higher homotopy groups are the homotopy groups of the universal
  covering, which in this case is a finite simply connected CW-complex. Note
  that a finitely generated abelian group $A$ with $A\tensor \rationals=0$ is
  automatically finite.
\end{proof}

\begin{remark}
  The general strategy leading toward Theorem \ref{theo:uniqueness} has been
  developed by Moritz Wiethaup 2006/07. However, his work has not been
  published yet. This was then taken up and developed further in
  \cite{BSuniqueness}. 
\end{remark}

\subsection{Uniqueness of differential K-theory}
\label{sec:uniqueness_of_K}

We observe that all the assumptions of Theorem \ref{theo:uniqueness} are
satisfied by K-theory, and also by real K-theory. Therefore, we have the
following theorem:

\begin{theorem}\label{theo:uniqueness_of_K}
  Given two differential extensions $\hat K$ and $\hat K'$ of complex
  K-theory, there is a unique natural isomorphism $\hat K^{ev}\to \hat
  {K'}^{ev}$ 
  compatible with all the structure. If the extensions are multiplicative,
  this   transformation is compatible with the products.

  If both extensions come with $S^1$-integration as in Definition
  \ref{def:integration_over_S1} there is a unique natural isomorphism $\hat
  K\to \hat K'$ compatible with all the structure, including the integration. 
\end{theorem}

In other words, all the different models for differential K-theory of Section
\ref{sec:models} define the same groups ---up to a  canonical
isomorphism.

\begin{remark}
  In Theorem \ref{theo:uniqueness_of_K} we really have to require the
  existence of $S^1$-integration. In \cite[Theorem 6.2]{BSuniqueness} an
  infinite family of ``exotic'' differential extensions of K-theory are
  constructed. Essentially, the abelian group structure is modified in a
  subtle way in these examples to produce non-isomorphic functors which all
  satisfy the axioms of Section \ref{sec:axioms}.
\end{remark}

\section{Models for differential K-theory}
\label{sec:models}

\subsection{Vector bundles with connection}
\label{sec:vect_with_conn}

The most obvious attempt to construct differential cohomology (at least for
$\hat K^0$) is to use vector bundles \emph{with connection}. It is technically
convenient to also add an odd differential form to the cycles. This, indeed is
the classical picture already used by Karoubi in \cite[Section 7]{MR913964}
for his definition of ``multiplicative K-theory'', which we would call the
flat part of differential K-theory.

\begin{definition}\label{def:vect_with_conn}
  A cycle for vector bundle K-theory $\hat K^0(M)$ is a triple
  $(E,\nabla,\omega)$, where $E$ is a smooth complex Hermitean vector bundle
  over $M$, 
  $\nabla$ a Hermitean connection on $E$ and $\omega\in\Omega^{odd}(M)/{\im(d)}$ a {class of 
a}  differential
  form of odd degree. The \emph{curvature} of a cycle essentially is defined
  as the Chern-Weil representative $\ch(\nabla):=\tr(e^{-\frac{\nabla^2}{2\pi
      i}})$ of the Chern character of $E$, computed 
  using the connection $\nabla$:
  \begin{equation*}
    R(E,\nabla,\omega):= \ch(\nabla)-d\omega.
  \end{equation*}

  \begin{remark}
    We require the use of Hermitean connections to obtain real valued
    curvature forms. Alternatively, one would have to use as target of
    $ch$ cohomology with complex instead of cohomology with real
    coefficients. A slightly more extensive discussion of this matter can 
    be found in \cite[Section 2]{MR1312690}. 
  \end{remark}

We define in the obvious way the sum
  $(E,\nabla_E,\omega)+(F,\nabla_F,\eta):= (E\oplus F,\nabla_E\oplus
  \nabla_F,\omega+\eta)$. Two cycles $(E,\nabla,\omega)$ and
  $(E',\nabla',\omega')$ are equivalent if there is a third bundle with
  connection $(F,\nabla_F)$ and an isomorphism $\Phi\colon E\oplus F\to
  E'\oplus F$ such that {
    \begin{equation*}
     \widetilde{\ch}(\nabla\oplus \nabla_F,\Phi^{-1}(\nabla'\oplus\nabla_F)\Phi) = \omega-\omega', 
  \end{equation*}
  where $
  \widetilde{\ch}(\nabla,\nabla')$ denotes the transgression Chern form between the two connections
 such that $\ch(\nabla)-\ch(\nabla')=  d\widetilde{\ch}(\nabla,\nabla')$}. \end{definition}

\begin{remark}
  In \cite[Section 9]{freed10:_index_theor_in_differ_k_theor}, a model for
  differential $\hat K^1$ is given where the cycles are Hermitean vector bundles
  with connection and a unitary automorphism, and an additional form {(modulo the image of $d$)} of even
  degree. The
  relations include in particular a rule for the composition of the
  unitary automorphisms: if $U_1$
  and $U_2$ are two unitary automorphisms of $E$ then
  $(E,\nabla,U_1,\omega_1)+(E,\nabla,U_2,\omega_2) 
  = (E,\nabla,U_2\circ U_1, CS(\nabla,U_1,U_2)+\omega_1+\omega_2)$, where $CS(\nabla,U_1,U_2)$ is
  the {Chern-Simons} form relating $\nabla$, $U_2\nabla U_2^{-1}$ and
  $U_2U_1\nabla U_1^{-1}U_2^{-1}$. We do not
  discuss this model in detail here. 
\end{remark}

\subsection{Classifying maps}
\label{sec:def_via_classifying_maps}

Hopkins and Singer, in their ground breaking paper \cite{MR2192936} give a
cocycle model of a differential extension $\hat E$ of any cohomology theory
(with transformation) $(E,\ch)$, based on
classifying maps. Here, for the construction of $\hat{{E}}^n(X)$ one has to
choose two fundamental ingredients:
\begin{enumerate}
\item a classifying space $X_n$ for $E^n$; note that no smooth structure for
  $X_{{n}}$ is required (and could be expected)
\item a cocycle $c$ representing $\ch_n$, so that, whenever $f\colon X\to X_n$
  represents a class $x\in E^n(X)$, then $f^*c$ represents $\ch(x)\in
  H^n(X;N)$. We can think of this cocycle as an $N$-valued singular cocycle,
  although variations are possible.
\end{enumerate}

A Hopkins-Singer cycle for $\hat E(X)$ then is a so called \emph{differential
  function}, which by definition is a triple $(f,h,\omega)$ consisting of a
continuous map $f\colon X\to X_n$, a closed differential $n$-form $\omega$ with
values in $N$ and an $(n-1)$-cochain $h$ satisfying
\begin{equation*}
  \delta h=\omega - f^*c.
\end{equation*}
In other words: $f$ is an explicit representative for a class $x\in E^n(X)$,
and we are of course setting $I(f,h,\omega)=[f]\in E^n(X)$. $\omega$
is a de Rham representative of $\ch(x)$, and we are setting
$R(f,h,\omega):=\omega$. This data actually gives two explicit representatives
for $\ch(x)$, namely $\omega$ and $f^*c$ (here, we have to map both $\omega$
and $f^*c$ to a common cocycle model for $H^n(X;N)$ like smooth singular
cochains. $\omega$ defines such a cocycle by the de Rham homomorphism
``integrate $\omega$ over the chain'', $f^*c$
by restriction).

By definition, $(f_1,h_1,\omega_1)$ and $(f_0,h_0,\omega_0)$ are equivalent
if $\omega_0=\omega_1$ and there is $(f,h,\pr^*\omega_1)$ on $X\times [0,1]$
which restricts to the two 
cycles on $X\times\{0\}$ and $X\times \{1\}$.

The advantage of this approach is its complete generality. A disadvantage is
that the cycles don't have a nice geometric interpretation. Moreover,
operations (like the addition and multiplication) rely on the choice of
corresponding maps between the classifying spaces realizing those. These maps
have then to be used to define the same operations on the differential
cohomology groups. This is
typically not very explicit. Moreover, properties like associativity,
commutativity etc.~will not hold on the nose for these classifying maps but
are implemented by homotopies which have to be taken into account when
establishing the same properties for the generalized differential
cohomology. This can quickly get quite cumbersome and we refrain from carrying
this out in any detail.

\subsection{Geometric families of elliptic operators}
\label{sec:geometric_families}

In \cite{bunke-2007}, a cycle model for differential K-theory (there called
``smooth K-theory'') is developed which is based on local index theory. In
spirit, it is similar to the passage of the classical model of K-theory via
vector bundles to the Kasparov KK-model, where all families of generalized
index problems are cycles.

Similarly, the cycles of \cite{bunke-2007} are geometric families of Dirac
operators. It is clear that a lot of differential structure has to be present
to obtain \emph{differential} K-theory, so the definition has to be more
restrictive than in Kasparov's model. There are many advantages of an approach
with very general cycles:
\begin{itemize}
\item first and most obvious, it is very easy to construct elements of
  differential K-theory if one has a broad class of cycles.
\item the approach allows for a unified treatment of even and odd degrees
\item the flexibilities of the cycles allows for explicit constructions in
  many contexts. In particular, it is easy to explicitly define the product
  and also the push-forward along a fiber-bundle.
\end{itemize}

It might seems as a disadvantage that one necessarily has a broad equivalence
relation and that it is therefore hard to construct homomorphisms out of
differential K-theory. To do this, one has to use the full force of local
index theory. However, this is very well developed and one can make use of
many properties as black box and then efficiently carry out the relevant
constructions. 

\begin{definition}\label{def:families_cycles}
Let $X$ be a compact manifold, possibly with boundary.
A cycle for a  $\hat K(X)$ is a pair
$(\cE,\rho)$, where $\cE$ is a geometric family, and $\rho\in
\Omega(X)/\im(d)$ is a class of differential forms.

A geometric family over $X$ (introduced in \cite{MR2191484}) consists of
the following data: 
\begin{enumerate}
\item a proper submersion with closed fibers $\pi\colon E\to X$,
\item a vertical Riemannian metric $g^{T^v\pi}$, i.e.~a metric on the vertical
  bundle $T^v\pi\subset TE$, defined as $T^v\pi:=\ker(d\pi\colon TE\to \pi^*TX)$.
\item a horizontal distribution $T^h\pi$, i.e.~a bundle $T^h\pi\subseteq
  TE$ such that $T^h\pi\oplus T^v\pi=TE$.
\item a family of Dirac bundles $V\to E$,
\item an orientation of $T^v\pi$.
\end{enumerate}
Here, a family of Dirac bundles consists of
\begin{enumerate}
\item a Hermitean vector bundle with connection $(V,\nabla^V,h^V)$ on $E$,
\item a Clifford multiplication $c\colon T^v\pi\otimes V\to V$,
\item on the components where
$\dim(T^v\pi)$ has even dimension a $\integers/2\integers$-grading $z$.
\end{enumerate}
We require that the restrictions of the family Dirac bundles to the fibers
$E_b:=\pi^{-1}(b)$, $b\in X$, give Dirac bundles in the usual sense (see
\cite[Def. 3.1]{MR2191484}):
\begin{enumerate}
\item The vertical metric induces the Riemannian structure on $E_b$,
\item The Clifford multiplication turns $V_{|E_b}$ into a Clifford module (see \cite[Def.3.32]{MR1215720}) which is graded if $\dim(E_b)$ is even.
\item The restriction of the connection $\nabla^V$ to $E_b$ is  a Clifford connection   (see \cite[Def.3.39]{MR1215720}). 
\end{enumerate}

A geometric family is called even or odd, if $\dim(T^v\pi)$ is
even-dimensional or odd-dimensional, respectively, and the form $\rho$ has the
corresponding opposite parity. 
\end{definition}

\begin{example}\label{ex:zero-dim}
  The cycles for $\hat K^0(X)$ of Definition \ref{def:vect_with_conn} are
  special cases of the cycles of Definition \ref{def:families_cycles}:
  $p\colon E\to X$ is just $\id\colon X\to X$, i.e.~the fibers consist of the
  $0$-dimensional manifold $\{pt\}$. 
\end{example}

There are obvious notions of isomorphism (preserving all the structure) and
of direct sum of cycles.
We now introduce {the structure maps $I$ and $R$ and the}  equivalence relation on the semigroup of isomorphism
classes. 

\begin{definition}\label{oppdef} The opposite $\cE^{op}$ of a geometric family
  $\cE$ is obtained by reversing 
the signs of the Clifford multiplication and the grading (in the even case) of
the underlying family of Clifford bundles, and of the orientation of the
vertical bundle.  
\end{definition}

\begin{definition}\label{def:I_fam}
  The usual construction of Dirac type operators of a Clifford bundle (compare
  \cite{MR1031992,MR1215720}), applied fiberwise, assigns to a geometric
  family $\cE$ over $X$ a family of Dirac type operators parameterized by $X$, and
  this is indeed the main idea behind the geometric families. Then, the
  classical construction of Atiyah-Singer assigns to this family its
  (analytic) index $\ind(\cE)\in K^*(B)$, where $*$ is equal to the parity of
  the dimension of the fibers. In the special case of Example
  \ref{ex:zero-dim} ---a vector bundle with connection over $X$--- this is
  exactly the K-theory class of the underlying $\integers/2$-graded vector
  bundle.

  We define $I(\cE,\omega):=\ind(\cE)\in K^*(X)$.
\end{definition}

\begin{remark}
  We define $\cE^{op}$ in such a way that
  $\ind(\cE^{op})=-\ind(\cE)$. Moreover, $\ind$ is additive under sums of
  geometric families. The
  equivalence relation we are going to define will be compatible with $\ind$.
\end{remark}

We now proceed toward the definition of $R(\cE)$. It is based on the notion of
local index form, an explicit de Rham representative of $\ch(\ind(\cE))\in
H^*_{dR}(X)$. It is one of the important points of the data collected in a
geometric family that such a representative can be constructed canonically. 
For a detailed definition we refer to \cite[Def.~4.8]{MR2191484}, but
we  briefly formulate its construction as follows.
The vertical metric $T^v\pi$ and the horizontal distribution $T^h\pi$ together
induce a connection $\nabla^{T^v\pi}$ on $T^v\pi$. 
Locally on $E$ we can assume that $T^v\pi$ has a spin structure. We let
$S(T^v\pi)$ be the associated spinor bundle. Then we can write the family of
Dirac bundles $V$ as  
$V=S\otimes W$ for a twisting bundle $(W,h^W,\nabla^W,z^W)$ with
metric, metric connection, and $\integers/2\integers$-grading which is determined uniquely
up to isomorphism.  
The form
$\hat A(\nabla^{T^v\pi}) \wedge \ch(\nabla^{W})\in \Omega(E)$ is globally defined, and we get the local index form by applying the integration over the fiber $\int_{E/B}\colon \Omega(E)\to \Omega(B)$:
\begin{equation}\label{eq:def_R_fam}
\Omega(\cE):=\int_{E/B}\hat A(\nabla^{T^v\pi}) \wedge \ch(\nabla^{W})\ .
\end{equation}

The characteristic class version of the index theorem for families is 
\begin{theorem} [\cite{MR0279833}]\label{thm2}
$\ch_{dR}(\ind(\cE))=[\Omega(\cE)]\in H^*_{dR}(X)$.
\end{theorem}
A proof using methods of local index theory has been given in \cite{MR813584},
compare \cite{MR1215720}.

The equivalence relation we impose is based on the following idea: a
geometric family $\cE$ with index $0$ should potentially be equivalent to the
cycle $0$, but in general only up to some differential form (with degree of
shifted parity). Moreover, in this case the local index form $R(\cE)$ will be
exact, but it is important to find an explicit primitive $\eta$ with
$d\eta=R(\cE)$. Therefore, we identify geometric \emph{reasons} why the index
is zero and which provide such a primitive.

\begin{definition}\label{def:taming}
  A pre-taming of $\cE$ is a family $(Q_b)_{b\in B}$ of self-adjoint operators $Q_b\in B(H_b)$
given by a smooth fiberwise integral kernel $Q\in C^\infty(E\times_B E,V\boxtimes V^*)$.
In the even case we assume in addition that $Q_b$ is odd with respect to the grading. 
The pre-taming is called a taming if $D(\cE_b)+Q_b$ is invertible for all
$b\in B$. In this case, by definition $\ind(D(\cE)+Q)$ is zero. However, as the
index is unchanged by the smoothing perturbation $Q$, also $\ind(\cE)=0$ if
$\cE$ admits a taming.
\end{definition}

{
Let $\cE_{t}$ be the notation for geometric family $\cE$ with a chosen taming.}  In
\cite[Def. 4.16]{MR2191484}, the  $\eta$-form
$\eta(\cE_{t})\in \Omega(X)$ is defined. By \cite[Theorem 4.13]{MR2191484}) it
satisfies  
 \begin{equation}\label{detad}
d\eta(\cE_{t})=\Omega(\cE)\ .
\end{equation} 
We skip the considerable analytic difficulties in the construction of the
eta-form and use it and its properties as a black box.

\begin{definition}\label{def:diff_K_families}
  We call two cycles $(\cE,\rho)$ and $(\cE^\prime,\rho^\prime)$ paired if there exists a taming
$(\cE\sqcup_B \cE^{\prime op})_t$ such that
$$\rho-\rho^\prime=\eta((\cE\sqcup_B \cE^{\prime op})_t)\ .$$
We let $\sim$ denote the equivalence relation generated by the relation
"paired".

We define the differential $K$-theory $\hat K^*(B)$ of $B$ to be the group completion of
the abelian semigroup of equivalence classes of cycles as in Definition
\ref{def:families_cycles} with fiber dimension congruent $*$ modulo $2$. 
\end{definition}

\begin{theorem}
  Definition \ref{def:diff_K_families} of $\hat K^*$ indeed defines (with the
  obvious notion of pullback) a contravariant functor on the category of
  smooth manifolds which is a differential extension of K-theory.

  The necessary natural transformations are induced by
  \begin{enumerate}
  \item $I\colon \hat K^*(X)\to K^*(X); \quad (\cE,\omega)\mapsto \ind(\cE)$  of
    Definition \ref{def:I_fam}
  \item $a\colon \Omega^{*-1}(X)/\im(d)\to \hat K^*(X); \quad \omega\mapsto (0,-\omega)$
  \item $R\colon \hat K^*(X)\to \Omega_{d=0}^*(X); \quad (\cE,\omega)\mapsto
    \Omega(\cE)+d\omega$\ of \eqref{eq:def_R_fam}.
  \end{enumerate}
\end{theorem}
\begin{proof}
  This is carried out in \cite[Section 2.4]{bunke-2007}. 
\end{proof}

\begin{definition}\label{def:produ_fam}
  Given two geometric families $\cE,\cF$ over a base $X$, there is an obvious geometric
  way to define their fiber product $\cE\times_B\cF$, with underlying fiber
  bundle the fiber product of bundles of manifolds etc. Details are spelled
  out in \cite[Section 4.1]{bunke-2007}. We then define the product of two
  cycles $(\cE,\rho)$ and $(\cF,\theta)$ (homogeneous of degree
$e$ and $f$, respectively) as
  \begin{equation*}
    (\cE,\rho)\cup (\cF,\theta):=[\cE\times_B\cF,(-1)^e\Omega(\cE)\wedge
\theta +\rho\wedge \Omega(\cF)-(-1)^ed\rho\wedge \theta]\ .
  \end{equation*}
\end{definition}

\begin{theorem}
  The product of Definition \ref{def:produ_fam} turns $\hat K^*$ into a
  multiplicative differential extension of K-theory.
\end{theorem}
\begin{proof}
  This is carried out in \cite[Section 4]{bunke-2007}.
\end{proof}

\begin{remark}
  Note that it is much more cumbersome to define a product structure with the
  other models described: for the vector bundle model the inclusion of the odd
  part is problematic, for the homotopy theoretic model one would have to
  realize the product structure using maps between (products of) classifying
  spaces. This is certainly a worthwhile task, but does not seem to be carried
  out in detail so far.

  In the geometric families model, it is also rather easy to construct
  ``integration over the fiber'' and in particular $S^1$-integration as
  defined in \ref{def:integration_over_S1}. This will be explained in Section
  \ref{sec:integration}. 
\end{remark}

\subsection{Differential characters}

One of the first models for a differential extension of integral cohomology
are the differential characters of Cheeger-Simons \cite{MR827262}. A
differential character is a pair $(\phi,\omega)$, consisting of a homomorphism
{ $\phi\colon{Z}_*(X)\to \reals/\integers$
  defined 
on the group of (smooth) singular cycles on $X$ such that for a boundary $b=dc$, $c\in C_{*+1}(X)$
 $\phi(b)=[\int_{c}\omega]_{\reals/\integers}$. }The set of differential characters on $C_*(X)$ is a
model for $\hat H^{*+1}(X)$. Along similar lines, in \cite{MR2231056} the cycle
model for K-homology due to Baum and Douglas \cite{MR679698} ---recently worked
out in detail in \cite{MR2330153}--- is used to define $\hat K^*(X)$ as the
group 
of $\reals/\integers$-valued homomorphism from the set of Baum-Douglas cycles
for K-homology which, on boundaries, are given as pairing with a differential
form. \cite{MR2231056} also define this pairing of a cycle and a differential
form. To do this, in contrast to \cite{MR679698} the cycles have to carry
additional geometric structure.

\subsection{Multiplicative K-theory}

The first work toward differential K-theory probably has been carried out by
Karoubi. We already mentioned his approach to $\reals/\integers$-K-theory
as the flat part of differential K-theory
of \cite{MR913964} with a model based on vector bundles with connection. 

This research has been deepened by Karoubi in subsequent publications
\cite{MR1076525,MR1273841}. The ``multiplicative'' K-theory introduced and
studied there (and related to work of Connes and Karoubi on the
multiplicative character of a Fredholm module \cite{MR972606}, which explains the terminology) is
associated to the usual K-theory and a filtration of the de Rham complex.

This theory looks very much like differential K-theory. With some
normalizations, it seems that one should be able to obtain differential
K-theory from the multiplicative one by considering the filtration of 
the de Rham complex given by the subcomplex of closed differential forms which
are Chern characters of vector bundles (with a connexion) on the given space. 
Karoubi uses multiplicative K-theory to construct characteristic
classes for holomorphic, foliated or flat bundles.

We thank Max Karoubi for explanations of his work on multiplicative
and differential K-theory.

\subsection{Currential K-theory}
\label{sec:subs-cycle-models}

Freed and Lott in \cite[2.28]{freed10:_index_theor_in_differ_k_theor}
introduce a variant of differential K-theory where they replace the
differential forms throughout by currents. This \emph{changes} the theory
in us much as the curvature homomorphism then also takes values in
currents, i.e.~forms are replaced by currents throughout. Because the
multiplication of currents is not well defined, in this new theory one 
loses the ring structure. The advantage of this variant, on the other hand, is
that push-forward maps can be described very directly also for embeddings: for
the differential form part, the image under push-forward will be a current
supported on the submanifold. The currential theory has the advantage that
push-forwards can be defined easily. Moreover, they can be defined for
arbitrary maps. However, pullbacks are not always defined (and not discussed
at all for this theory in
\cite{freed10:_index_theor_in_differ_k_theor}). Consequently, the theory
should (up to a degree shift given by Poincar\'e duality) better be considered
as differential K-homology, twisted by some kind of $\spinc$-orientation
twist.

To describe differential extensions of bordism theories as in
\cite{MR2550094,bunke10:_geomet_descr_of_smoot_cohom}, currents are also used
---but there only as an intermediate tool. The axiomatic setup for those
theories is the one described in \ref{def:axioms}. In particular, the
curvature homomorphism takes values in smooth differential forms.

\subsection{Differential K-theory via bordism}

In \cite[Section 4]{MR2550094}, a geometric model of the differential extension
$\hat{MU}^*$  of $MU^*$ is constructed, where $MU^*$ is 
cohomology theory dual to complex bordism. Cycles are
given by pairs $(\tilde c,\alpha$). Here $\tilde c$ is a geometric cycle for
$MU^n(X)$, {i.e.~a proper smooth map $W\to X$ with $n=\dim(X)-\dim(W)$ and
an  
explicit geometric model of the stable normal bundle with $U(N)$-structure. Moreover, }
$\alpha$ is a differential $(n-1)$-form with distributional coefficients whose
differential differs from a certain characteristic current of $\tilde c$ by a
smooth differential form. We impose on these cycles the equivalence relation
generated by an obvious notion of bordism together with stabilization of the
normal bundle data.

The functor $\hat MU^*$ is a multiplicative extension of complex cobordism. In
particular, it takes values in algebras over $\hat{MU}^{ev}(*)$, as $\hat
MU^{ev}(*)$ is a subring of $\hat MU^*(*)$. However, as $MU^{odd}(*)=0$, the
long exact sequence \eqref{exax} give the canonical isomorphism
\begin{equation*}
\hat
MU^{ev}(*)\xrightarrow{\iso} MU^{ev}(*)= MU^*(*).
\end{equation*}

Assume now that $E^*$ is a generalized cohomology theory which is
\emph{complex oriented}, i.e.~which comes with a natural transformation
$MU^*\to E^*$. An example is complex K-theory, with the usual orientation
$MU\to MSpin^c\to K$. If the complex oriented theory $E^*$ is  Landweber exact (i.e. the condition of 
\cite[Theorem $2.6_{MU}$]{Landweber76} is satisfied),
then Landweber \cite[Corollary 2.7]{Landweber76} proves that $E^*(X)\iso
MU^*(X)\tensor_{MU}E^{*}$ is obtained from $MU^*(X)$ by just taking the tensor
product with the coefficients $E^*$ (in the graded sense). Complex K-theory is
Landweber exact \cite[Example 3.3]{Landweber76} so that this principle
gives a simple bordism definition of K-theory. 

In \cite[Theorem 2.5]{MR2550094} it is shown that Landweber's results extends
directly to differential extensions. If $E^*$ is a Landweber exact complex
oriented generalized cohomology theory, then
\begin{equation*}
  \hat E^*(X):= \hat{MU}^*(X)\tensor_{MU}E
\end{equation*}
is a multiplicative differential extension of $E^*(X)$. Moreover, as explained
above, we 
can apply this to complex K-theory and obtain a bordism description of
differential K-theory:
\begin{equation*}
  \hat K^*(X) = \hat MU^*(X)\tensor_{MU} K.
\end{equation*}
By \cite[Section 2.3]{MR2550094}, this is indeed a multiplicative differential
extension of complex K-theory. Furthermore, it has a natural $S^1$-integration
as in Definition \ref{def:integration_over_S1} and even general ``integration
over the fiber'' as discussed in Section \ref{sec:integration}.

\subsection[Geometric cycle models for $\reals/\integers$-K-theory]{Geometric
  cycle models for $\reals/\integers$-K-theory via the flat 
  part of differential K-theory}

As K-theory satisfies the conditions of Theorem \ref{theo:integration_exists},
any {multiplicative} differential extension canonically comes with a transformation
``integration over $S^1$''. Therefore, by Theorem \ref{theo:flat_is_cohom} {its flat subfunctor} is
a generalized cohomology theory. Moreover, K-theory satisfies the conditions
of Theorem \ref{theo:flat_is_RZ} and we therefore get a natural isomorphism
\begin{equation*}
  \hat K_{flat}\to K\reals/\integers.
\end{equation*}

In other words, any of the models for differential K-theory described above
provides a model for $\reals/\integers$-K-theory. In particular, we recover
the original result \cite[Section 7.21]{MR913964} that the flat part of the
vector bundle model of Section \ref{sec:vect_with_conn} describes
$K\reals/\integers^{-1}$. 

\subsection{Real K-theory}

It {seems to be} not difficult to modify the above models to obtain differential
extensions $\hat{KO}$ of real K-theory. This is particularly clear with the vector bundle
model of Section \ref{sec:vect_with_conn}. Complex Hermitean vector bundles
with Hermitean connection have to be replaced by real Euclidean vector bundles
with metric connection. This describes $\hat{KO}^0$. Of course, $\hat{KO}^*$
now is $8$-periodic. The full model in terms of geometric families, following
Section \ref{sec:geometric_families} replaces the families of Dirac
bundles by their real analogs. {For the analysis of $\eta$-forms it seems to be most suitable to implement the real structures via additional conjugate linear operations on the complex Dirac bundle (as explained e.g.~in \cite[Sec. 2.2]{MR2443109}). Alternatively one could possibly work with 
$\dim(T^v\pi)$-multigradings in the sense of e.g.~\cite[A.3]{MR1817560} (in
other words, a compatible $Cl_k$-module structure, compare also \cite[Section
16]{MR1031992}).  The details of a model for $\hat{KO}$ based on geometric families have still to be worked out. }
 It is actually worthwhile to work out a Real differential
K-theory in the usual sense (i.e.~for spaces with additional
involution). However, we will not carry this out here.

\subsection{Differential K-homology and bivariant differential K-theory}

An important aspect of any cohomology theory is the dual homology theory and
the pairing between the cohomology theory and the homology theory. This
applies in particular to K-theory, where the dual homology theory can be
described with three different flavors:
\begin{itemize}
\item homotopy theoretically, using the K-theory spectrum
\item with a geometric cycle model introduced by Baum and Douglas
  \cite{MR679698}
\item an analytic model proposed by Atiyah \cite{MR0266247} and made precise
  by Kasparov \cite{MR0488027}, compare also \cite{MR1817560}.
\end{itemize}

\begin{definition}
  The cycles for $K_*(X)$ are triples $(M,E,f)$ where $M$ is a closed manifold
  with a given $\spinc$-structure with dimension congruent to $*$ modulo $2$,
  $f\colon M\to X$ is a continuous map and $E\to X$ is a complex vector
  bundle. On the set of isomorphism classes of triples we put the equivalence
  relation generated by bordism, \emph{vector bundle modification} and the relation that $(M,E_1,f)\disjointunion
  (M,E_2,f)\sim (M,E_1\oplus E_2,f)$.

Here, vector bundle modification means that, given a complex vector bundle
$V\to M$, $(M,E,f)\sim (S(V\oplus\reals),T(E),f\circ p)$ where $S(V\oplus
\reals)$ is the sphere bundle of the real bundle $V\oplus \reals$, $p\colon
V\oplus\reals\to M$ is the bundle projection and $T(E)$ is a vector bundle
which (up to addition of a bundle which extends over the disc bundle)
represents the push-forward of $E$ along the ``north pole inclusion'' $i\colon
M\to S(E\oplus \reals); x\mapsto (x,0,1)$. There is an explicit clutching
construction of this bundle.
\end{definition}

It took  over 20 years before in \cite{MR2330153} the equivalence of the
geometric cycle model with the other models has been worked out in full
detail.

In the geometric model, the pairing between K-homology and K-theory has a very
transparent description, related to index theory. Given a K-homology cycle
$(M,E,f)$ and a K-theory cycle represented by the vector bundle $V\to X$, the
pairing is the index of the $\spinc$ Dirac operator $D_{E\tensor f^*V}$ on
$M$,  twisted with $E\tensor f^*V$. In the analytic model, this Dirac operator
by itself essentially already is a K-homology cycle.

The analytic theory has two very important extensions: first of all, it
extends from a (co)homology theory for spaces to one for arbitrary
$C^*$-algebras. ``Extension'' here in the sense that the category of compact
Hausdorff spaces is equivalent to the opposite category of commutative unital
$C^*$-algebras via the functor $X\mapsto C(X)$. Secondly, Kasparov's theory
really is bivariant, with $KK(X,*)=K_*(X)$ being K-homology, and
$KK(*,X)=K^*(X)$ K-theory. Most important is that KK-theory comes
with an associative composition product $KK(X,Y)\tensor KK(Y,Z)\to
KK(X,Z)$. This is an extremely powerful tool with many applications in index
theory and beyond.

In this context, it is very desirable to have differentiable extensions also
of K-homology and ideally of bivariant KK-theory (of course only restricted to
manifolds). This should e.g.~provide a powerful home for refined index
theory. 

\begin{definition}\label{def:bivariant}
  Indeed, guided by the model of differential K-theory, we see that
  differential K-homology or more general $\hat{KK}$, differential bivariant
  KK-theory, should be a 
  functor from smooth manifolds to graded abelian groups with the following
  additional structure:
  \begin{enumerate}
  \item A transformation $I\colon \hat{KK}(X,Y)\to KK(X,Y)$
  \item a transformation $R$ to the appropriate version of differential
    forms. For K-homology, the best bet for this are differential currents
    (essentially the dual space of differential forms) and for the bivariant
    theory the {closed (i.e. commuting with the differentials)}
    continuous homomorphisms from differential forms of the first space to
    differential forms of the second: $R\colon \hat{KK}(X,Y)\to
    \Hom(\Omega(X),\Omega(Y))$.
  \item action of ``forms'': a {degree-$1$} transformation $a\colon
    \Hom(\Omega(X),\Omega(Y))\to \hat{KK}(X,Y)$. 
  \end{enumerate}
  These have to satisfy relations and exact sequences which are direct
  generalizations of those of Definition \ref{def:axioms}:
   \begin{equation*}
  \hat{KK}(X,Y)\xrightarrow{R}[-1]
  \Hom(\Omega^*(X),\Omega^*(Y))[-1]/\im(d)\xrightarrow{a}
  \hat{KK}(X,Y)\xrightarrow{I} 
    KK(X,Y)\to 0
  \end{equation*}
  \begin{equation*}
    \alpha\circ R= d.
  \end{equation*}

  Here, we adopt the tradition that KK-theory is $\integers/2\integers$-graded
  and use the even/odd grading of forms throughout.

  Additionally, there should be a natural associative ``composition product''
  $\hat{KK}(X,Y)\tensor \hat{KK}(Y,Z)\to \hat{KK}(X,Z)$ such that $I$ maps
  this product to the Kasparov product and $R$ to the composition.
\end{definition}

Note that (up to Poincar\'e duality degree shifts), the currential K-theory of
Section \ref{sec:subs-cycle-models} actually is a model for differential
K-homology, which is the specialization of the bivariant theory to
$\hat{KK}(X,*)$.

An early preprint version of \cite{bunke-2007} contains a section which
develops a bivariant theory as in Definition \ref{def:bivariant} and its
applications. The groups are defined in a cycle model where a cycle $\phi$ for
$KK(X,Y)$ consists of
\begin{enumerate}
\item a $\hat K$-oriented bundle $p\colon W\to X$, using $\hat K$-orientations
  as in Section \ref{sec:orientations},
\item a class $x\in \hat K(W)$,
\item a smooth map $f\colon W\to Y$
\item a continuous homomorphism $\Phi\in\Hom(\Omega^*(X),\Omega^*(Y))[-1]$.
\end{enumerate}

On the set of isomorphism classes of cycles one puts the equivalence relation
generated by
\begin{enumerate}
\item compatibility of the sum operations (one given by disjoint union, the
  other by addition in $\hat K(W)$ and $\Hom(\Omega^*(X),\Omega^*(Y))$).
\item bordism
\item a suitable definition of vector bundle modification,
\item change of $\Phi$ by the image of $d$.
\end{enumerate}
Details of such a bivariant construction, or a construction of differential
K-homology seem not to exist in the published literature.
{Note that this approach is not tied to $K$-theory. It provides a construction of a bivariant
differential theory from a differential extension of a cohomology theory together with a theory of differential orientation and integration satisfying a natural set of axioms. }

\section{Orientation and integration}
\label{sec:integration}

\subsection{Differential K-orientations}
\label{sec:orientations}

Let $p\colon W\to B$ a proper submersion with fiber-dimension $n$. An
important aspect of cohomology is a map ``integration along the fiber'' or
``push-forward'' $p_!\colon E^*(W)\to E^{*-n}(B)$. There is a general theory
for this, and it requires the extra structure of an \emph{$E$-orientation} of $p$
(one could also say an $E$-orientation of the fibers of $p$). If $E$ is
ordinary cohomology, this comes down to an ordinary orientation and in de Rham
cohomology $p_!$ literally is ``integration along the fiber''.

For K-theory, it is a $\spinc$-structure 
of $\pi$ which gives rise to a K-theory orientation, which in turn defines a
topological push-forward map $p_!\colon K^*(W)\to K^{*-n}(B)$.

To extend this to \emph{differential} K-theory, one has to add additional
geometric data to the $\spinc$-structure.

This is a pattern which applies to a general differential cohomology theory
$\hat E$:
a differential orientation will consist of an $E$-orientation in the usual
sense together with extra differential and geometric data. 
There is one notable exception, though. In ordinary integral differential
cohomology (Cheeger-Simons differential characters), a topological
orientation (which is an ordinary orientation) lifts uniquely to a
differential orientation. Therefore, in the literature treating ordinary
differential cohomology and push-forward in that context, differential
orientations are not discussed \cite{MR2179587,MR1772521}.

We now describe a model for differential K-orientations for a proper
submersion $p\colon W\to B$ which particularly suits the analytic model for
$\hat K$ of Section \ref{sec:geometric_families}.  {
We will describe the $\integers$-graded version of orientations in order to
make clear in which precise degrees the form constituents live. This is
important if one wants to set up a similar theory in the case of other, non-two-periodic theories like complex bordism.}

Fix an 
 underlying topological K-orientation of $p$ by choosing a $\spinc$-reduction
 of the structure group 
 of the vertical tangent bundle $T^vp$ of $p$. In order to make this precise we
 choose along the way an orientation and a metric on $T^vp$. 

We now consider the set $\cO$ of tuples
$(g^{T^vp},T^hp,\tilde \nabla,\sigma)$ where
\begin{enumerate}
\item $g^{T^vp}$ is the Riemannian metric on the vertical bundle $\ker Tp$
  (and the K-orientation is given by a $\spinc$-reduction of the resulting
  $O(n)$-principal bundle $P$, i.e.~a $\spinc$-principal bundle $Q$ lifting $P$).
\item $T^hp$ is a horizontal distribution.
\item From the horizontal distribution, we get a connection
$\nabla^{T^vp}$ which restricts to the Levi-Civita connection along the
fibers as follows. First one chooses a metric $g^{TB}$
on $B$. It induces a horizontal metric $g^{T^hp}$ via the isomorphism
$dp_{|T^hp}\colon T^hp\stackrel{\sim}{\to}
p^*TB$.  We get a metric $g^{T^vp}\oplus g^{T^hp}$ on $TW\cong T^vp\oplus
T^hp$ which gives rise to a Levi-Civita connection. Its projection to $T^vp$
is $\nabla^{T^vp}$. $\tilde \nabla$ is a $\spinc$-reduction of
$\nabla^{T^vp}$, i.e.~a connection on the $\spinc$-principal bundle $Q$ which reduces
to  $\nabla^{T^vp}$.  If we think of the connections $\nabla^{T^vp}$ and
$\tilde \nabla$ in terms of horizontal distributions $T^hSO(T^vp)$ and $T^hQ$,
then we say that  $\tilde \nabla$ reduces to $\nabla^{T^vp}$ if
$d\pi(T^hQ)=\pi^*(T^hSO(T^vp))$.
\item $\sigma\in \Omega^{-1}(W;\reals[u,u^{-1}])/\im(d)$.  \end{enumerate}

We introduce the globally defined complex line bundle
\begin{equation}\label{eq:def_of_L2}
L^2:=Q\times_{\lambda}\complexs \to W
\end{equation}
 associated to
the $\spinc$-bundle $Q$ via the representation $\lambda\colon Spin^c(n)\to
U(1)$. The connection $\tilde \nabla$ thus induces a connection on
$\nabla^{L^2}$.

We introduce the form 
\begin{equation} \label{uzu111} c_1(\tilde \nabla):=\frac{1}{2}(2\pi i u )^{-1}R^{L^2}  {\in  \Omega^{0}(W;\reals[u,u^{-1}])}.\end{equation} 
Let $R^{\nabla^{T^vp}}\in \Omega^2(W,\End(T^vp))$ denote the curvature of $\nabla^{T^vp}$.
The closed form
$$\hA(\nabla^{T^vp}):={\det}^{1/2}\left(
\frac{
\frac{ u^{-1}R^{\nabla^{T^vp}}}{4\pi}}
{\sinh\left(\frac{u^{-1}R^{\nabla^{T^vp}}}{4\pi}\right)}\right){\in  \Omega^{0}(W;\reals[u,u^{-1}])}
$$
represents the $\hA$-class of $T^vp$. 
\begin{definition}\label{uzu3} The relevant differential form for local index theory
  in the $\spinc$-case is
  \begin{equation}\label{eq:A_hat_c}
\hA^c(\tilde \nabla):=\hA(\nabla^{T^vp})\wedge e^{c_1(\tilde \nabla)}\ .
\end{equation}
\end{definition}
If we consider $p\colon W\to B$ with the geometry $(g^{T^vp},T^hp,\tilde \nabla)$ and the Dirac bundle
$S^c(T^vp)$ as a geometric family $\cW$ over $B$, then {
$$\int_{W/B} \hA^c(\tilde \nabla)=\Omega(\cW)\in \Omega^{-n}(B;\reals[u,u^{-1}])\ ,$$} is the local index form of $\cW$. 

We introduce a relation $\sim$ on $\cO$:
Two tuples
$(g_i^{T^vp},T_i^hp,\tilde \nabla_i,\sigma_i)$, $i=0,1$ are related if and only if
$\sigma_1-\sigma_0=\tilde \hA^c(\tilde \nabla_1,\tilde \nabla_0)$, the
transgression form of $\hA^c(\tilde\nabla)$. In
\cite[Section 3.1]{bunke-2007}, it is shown that
$\sim$ is an equivalence relation.

\begin{definition}\label{smmmmzuz}
The set of differential $K$-orientations which refine a fixed underlying topological
$K$-orientation
of $p\colon W\to B$ is the set of equivalence classes $ \cO/\sim$.
{The vector space }
$\Omega^{-1}(W;\reals[u,u^{-1}])/\im(d)$ acts on the set of differential
$K$-orientations by translation of the differential form entry.
\end{definition}

In \cite[Corollary 3.6]{bunke-2007} we show:
\begin{proposition}\label{corol:orient_as_torsor}
The set of differential $K$-orientations refining a fixed underlying topological $K$-orientation is a torsor over
$\Omega^{-1}(W;\reals[u,u^{-1}])/\im(d)$.
\end{proposition}

\begin{remark}
  Proposition \ref{corol:orient_as_torsor} should generalize to differential
  extensions of a general cohomology theory $E$: the differential orientations
  refining a fixed $E$-orientation should form a torsor over
  $\Omega^{-1}(W;E\reals)$. If we apply this to the special case of ordinary
  cohomology, this group is trivial as the coefficients are concentrated in
  degree $0$. This explains why there is a unique lift of an ordinary
  orientation to a differential orientation in case of ordinary cohomology.
\end{remark}

\subsection{Integration}
\label{sec:integration_map}

We have set up our model for differential K-orientations in such a way that we
have a natural description of the integration homomorphism in differential
K-theory using the analytic
model of Section \ref{sec:geometric_families}. { 
From now on we reduce to the two-periodic case (by setting $u=1$).
} 

Given $p\colon W\to B$ with a differential K-orientation as in Section \ref{sec:orientations} and a
class $x=[\cE,\rho]\in \hat K(W)$, the basic idea for the construction of
$p_!(x)\in \hat K(B)$ is to
form the representative for $p_!(x)$ as the geometric family $p_!\cE$ obtained by
simply composing the family over $W$ with $p$ to obtain a family over $B$. The
geometry on $p$ given by the differential orientation allows to put the required geometry on the composition in a
canonical way. For example, the fiberwise metric is obtained as direct sum of
the fiberwise metric on $\cE$ and on the fibers of $p$. We omit the details which are given in \cite[Section
3.2]{bunke-2007}.

 The
main remaining question is to define the correct differential form; this
requires the study of adiabatic limits.
Indeed, in the construction of the geometry on $p_!\cE$ we can introduce an
additional parameter $\lambda\in (0,\infty)$ by scaling the metric on the
fibers of $\cE$ by $\lambda^2$ (and adjusting the remaining geometry) to
obtain $p_!^\lambda\cE$. In
total, this gives an adiabatic deformation family $\cF$ over $(0,\infty)\times
B$ which restricts to $p_!^\lambda\cE$ on $\{\lambda\}\times B$. 

Although the vertical metrics of $\cF$ and $p^\lambda_!\cE$ collapse as $\lambda\to 0$ the induced connections and the curvature tensors
on the vertical bundle $T^vq$ converge and simplify in this limit. This fact
is heavily used in local index theory, and we refer to \cite[Sec 10.2]{MR1215720}
for details. In particular, the integral 
\begin{equation}\label{dtqdutwqdwqdq}
\tilde \Omega(\lambda,\cE):=\int_{(0,\lambda)\times B/B}\Omega(\cF)\end{equation}
converges.
% , and we have
% \begin{equation}\label{eq88}
% \lim_{\lambda\to 0}\Omega(p_!^\lambda\cE)=\int_{W/B}\hA^c(\tilde\nabla)\wedge
% \Omega(\cE)\ ,\quad \Omega(p_!^\lambda\cE)-\int_{W/B}\hA^c(\tilde\nabla)\wedge
% \Omega(\cE)=d\tilde \Omega(\lambda,\cE)\ .
% \end{equation}

We now define
\begin{equation}\label{eq300}
\hat p_!(\cE,\rho):=[p_!\cE,\int_{W/B}  \hA^c(\tilde\nabla)\wedge \rho + \tilde
\Omega(1,\cE)+\int_{W/B}\sigma \wedge R([\cE,\rho])]\in \hat K(B)\ .
\end{equation}

This push-forward has all the expected properties, listed below. The proofs use the explicit
model and a heavy dose of local index theory {in order to verify that this  cycle-level construction is compatible with the equivalence relation involving tamings and $\eta$-forms}. They can be found in
\cite[Sections 3.2, 3.3, 4.2]{bunke-2007}.

\begin{theorem}\label{theo:properties_of_integration}
  The push-forward/integration for differential K-theory defined in
  \eqref{eq300} has the following properties.  
  \begin{enumerate}
  \item Given a proper submersion $p\colon W\to B$ with differential
    K-orientation and with fiber dimension $n$, the differential push-forward
    is $\hat p_!\colon \hat K^*(W)\to \hat K^{*-n}(B)$ a well defined
    homomorphism which only depends on the differential orientation, not its
    particular representative.
  \item If $q\colon Z\to W$ is another proper submersion with differential
    K-orientation, there is a canonical way to put a composed differential
    K-orientation on the composition $p\circ q$, and push-forward is
    functorial: $\hat{(p\circ q)}_! = \hat p_!\circ \hat q_!$.
  \item Fix a Cartesian diagram
    \begin{equation*}
      \begin{CD}
        W'  @>F>> W\\
        @VV{f^*p}V       @VVpV\\
        B' @>f>> B
      \end{CD}
    \end{equation*}
    where as before $p\colon W\to B$ is a proper submersion with differential
    K-orientation. There is a canonical way to pull back the differential
    K-orientation of $p$ to $p':=f^*p$. One then has compatibility of
    pull-back and push-forward
    \begin{equation*}
      f^*\hat p_! = \hat p'_! F^*.
    \end{equation*}
  \item The projection formula relating push-forward and product holds in
    differential K-theory
    \begin{equation*}
      \hat p_!(p^*y\cup x) = y\cup \hat p_!(x);\qquad x\in \hat K(W),\; y\in
      \hat K(B).
    \end{equation*}
  \item The differential push-forward is via the structure maps compatible
    with the push-forward in ordinary K-theory and on differential
    forms. However, for the differential forms we have to define the modified
    integration, depending on the differential K-orientation $o$, as
    \begin{equation*}
      p^o_!\colon \Omega(W)\to \Omega(B);\; \omega\mapsto \int_{W/B} (
      \hA^c(\tilde\nabla)-d\sigma)\wedge \omega.
    \end{equation*}
    Indeed, this does not depend on the representative of the differential
    K-orientation. Moreover, it induces $p^o_!\colon \Omega(W)/\im(d)\to
    \Omega(B)/\im(d)$. With this map, we get commutative diagrams
    \begin{equation}\label{uppersq}
      \begin{CD}
        K(W) @>{\ch}>> \Omega(W)/\im(d) @>a>>\hat K(W) @>{I}>> K(W)\\
        @VV{p_!}V     @VV{p^o_!}V     @VV{\hat p_!}V     @VV{p_!}V \\
        K(B) @>{\ch}>> \Omega(B)/\im(d) @>a>>\hat K(B) @>{I}>> K(B)\\
      \end{CD}
    \end{equation}
    \begin{equation}\label{lowersq}
      \begin{CD}
        \hat K(W) @>{R}>> \Omega_{d=0}(W)\\
        @VV{\hat p_!}V      @VV{p^o_!}V \\
        \hat K(B) @>{R}>> \Omega_{d=0}(B)
      \end{CD}
    \end{equation}
  \end{enumerate}
\end{theorem}
\subsection{$S^1$-integration}
\label{sec:suspension}

We consider the projection $\pr_1\colon B\times S^1\to B$. 

The projection $\pr_1$  fits into the Cartesian diagram
\begin{equation*}
  \begin{CD}
   B\times S^1 @>{\pr_2}>> S^1\\
    @VV{\pr_1}V @VV{p}V\\
    B @>r>> \{*\}\ .  \end{CD}
\end{equation*}

We choose the metric $g^{TS^1}$ of unit volume and the  bounding spin structure
on $TS^1$. 
This spin structure induces a $\spinc$ structure on $TS^1$ together with the connection $\tilde \nabla$.
In this way we get a representative $o$ of a differential $K$-orientation of $p$. 
By pull-back we get the representative $r^*o$ of a differential $K$-orientation of
$\pr_1$ which is used to define $(\hat\pr_1)_!$.

\begin{definition}
  We define $S^1$-integration for differential K-theory as in Definition \ref{def:integration_over_S1}
  simply by setting
  \begin{equation*}
    \int_{B\times S^1/B}:=(p_1)_! \colon \hat K^*(B\times S^1)\to \hat
    K^{*-1}(B)
  \end{equation*}
  where we use the differential K-orientation of $p_1\colon B\times S^1\to B$
  just described. Note that by Theorem \ref{theo:properties_of_integration} it
  has the properties required of $S^1$-integration.
\end{definition}

By \cite[Corollary 4.6]{bunke-2007} we get
\begin{equation*}
  (\hat \pr_1)_!\circ \pr_1^* =0.
\end{equation*}

\section{Index theory and natural transformations}

It is well established that index theory of elliptic operators is closely
related to K-theory and K-homology. Indeed, this is the reason why the
analytic model of Section \ref{sec:geometric_families} can work at all. For a
fiber bundle $p\colon W\to B$ with 
K-orientation, the push-forward in K-theory can be interpreted as the family
index of the fiberwise $\spinc$ Dirac operator, twisted with the bundle
representing the K-theory class.
Of course, one might define the push-forward in a different way ---then this
is a somewhat abstract statement.

Continuing on this formal level, the Chern character provides a way to compute
K-theory in terms of cohomology. Now there is also the push-forward in
cohomology. It is well known that Chern character and push-forward are not
compatible. The Riemann-Roch formula provides the {appropriate correction}. In some
sense, a Riemann-Roch theorem therefore is a cohomological index theorem.

On this level, we will now find a lift of index theory to differential
K-theory. In Section \ref{sec:integration} we have discussed the push-forward in
differential K-theory. We will now describe how to lift the Chern character to
a natural transformation from differential K-theory to differential cohomology
and will then discuss a differential Riemann-Roch theorem correcting the
defect that this Chern character is not compatible with integration. 

A further refinement of this theorem is obtained by a direct analytic definition
of a (family) index with values in differential K-theory, given in a natural
way for geometric families of elliptic index problems. The goal of an index
theorem is then to find a topological formula for this index, i.e.~a formula
which does not involve the explicit solution of differential equations. This
has indeed been achieved by Lott and Freed in
\cite{freed10:_index_theor_in_differ_k_theor} and we discuss the details in
Section \ref{sec:diff_AS}.

\subsection{Differential Chern character}

The classical Chern character has two fundamental properties. First, it is a
certain characteristic class of vector bundles. As such, it is a certain
explicit (rational) polynomial $p_{\ch}(c_1,c_2,\dots)$ in the Chern classes
of the vector bundle. Secondly, it tuns out that this characteristic class is
compatible with direct sum and stabilization and with Bott periodicity in the
appropriate way to define a natural transformation of cohomology theories
\begin{equation*}
  \ch\colon K^*{(\cdot)} \to H^*(\cdot;\rationals[u,u^{-1}]);\qquad \text{($u$ of
    degree $2$)}.
\end{equation*}
Finally, after tensor product with $\rationals$, $\ch\tensor\id_\rationals$
becomes an isomorphism {for finite $CW$-complexes}.

We demand that the Chern character for differential K-theory should display
the same properties: be implemented as characteristic class of vector bundles
with connection, and then pass to a natural transformation between
differential cohomology theories.

\subsubsection{Characteristic classes in differential cohomology}
\label{sec:diff_charac_class}

Early on, differential cohomology is closely related to
characteristic classes of vector bundles with connection. Indeed, $\hat
H^2(X;\integers)$ even is isomorphic to the set of isomorphism classes of
complex line
bundles with Hermitean connection. This isomorphism is implemented by the
differential version of the first Chern class.

Elaborating on this, Cheeger and Simons \cite[Section 4]{MR827262} construct
differential Chern classes of vector bundles with Hermitean connection with
values in integral differential cohomology
\begin{equation*}
  \hat{c}_k(E,\nabla) \in \hat H^{2k}(X;\integers)
\end{equation*}
with the following properties:
\begin{enumerate}
\item \label{item:ck_nat}naturality under pullback with smooth maps;
\item\label{item:ck_comp} compatibility with the classical Chern classes:
  \begin{equation*}
    I(\hat c_k(E,\nabla)) = c_k(E)\in H^{2k}(X;\integers);
  \end{equation*}
\item\label{item:CW_comp} compatibility with Chern-Weil theory
  \begin{equation*}
    R(\hat c_k(E,\nabla)) = C_k(\nabla^2) \in \Omega^{2k}(X)
  \end{equation*}
  where $C_k(\nabla)$ is the Chern-Weil form of the connection $\nabla$
  associated to the invariant polynomial $C_k$ which is (up to a scalar) the
  elementary symmetric polynomial in the eigenvalues.
\end{enumerate}
Moreover, Cheeger and Simons show that the differential Chern classes are
uniquely determined by these requirements. The proof is reminiscent of (and
indeed inspired) the proof of uniqueness in Section \ref{sec:Uniqueness}. They
use the universal example of $\complexs^n$-vector bundles with connection
---known to exist by \cite{MR0133772}. Once the differential Chern classes are
chosen for the universal example, they are determined by naturality for every
$(E,\nabla)$. The integral cohomology of the base space $BU(n)$ of the
universal example is concentrated in even degrees. The long exact
sequence~\eqref{eq:les_R_surj} then implies that $I\oplus R\colon \hat
H^{2k}(BU(n);\integers)\oplus \Omega_{d=0}^{2k}(BU(n))$ is injective, so that 
\ref{item:ck_comp} and \ref{item:CW_comp} determine the universal $\hat c_k$,
and they exist by the defining properties of $C_k(\nabla^2)$. One only has to
check that the construction is independent of the choice of the universal
model, which essentially follows from uniqueness. Note that $BU(n)$ is not
itself a finite dimensional manifold so that one has (as usual) to work with
finite dimensional approximations.

Differential ordinary cohomology is defined with coefficients in any subring
of $\reals$, and the inclusion of coefficient rings induces natural maps with
all the expected compatibility relations. In particular, we have differential
cohomology with coefficients in $\rationals$, $\hat H^*(\cdot;\rationals)$,
taking values in the category of $\rationals$-vector spaces, and commutative diagrams
\begin{equation*}
  \begin{CD}
    \hat H(\cdot;\integers) @>>> \hat H(\cdot;\rationals)\\
    @VV{R}V  @VV{R}V\\
    \Omega(\cdot) @>{=}>> \Omega(\cdot)
  \end{CD};\qquad
  \begin{CD}
    \hat H(\cdot;\integers) @>>>\hat H(\cdot;\rationals)\\
    @VV{I}V @VV{I}V\\
    H^*(\cdot;\integers) @>>> H^*(\cdot;\rationals).
  \end{CD}
\end{equation*}
Expressing the classical Chern character (uniquely) as a rational polynomial
in the Chern classes
\begin{equation*}
  \ch(E) = P_{\ch}(c_1(E),c_2(E),\dots);\quad\text{with } P_{\ch}\in\rationals[[x_1,x_2,\dots]],
\end{equation*}
we now define the differential Chern character
\begin{equation}\label{eq:def_of_diff_Chern_char}
  \hat{\ch}(E,\nabla):= P_{\ch}(\hat c_1(E,\nabla),\dots)\in \hat H^{2*}(X;\rationals).
\end{equation}

By construction, this is a natural characteristic class satisfying
\begin{equation*}
  \begin{split}
    I(\hat{\ch}(E,\nabla)) &= \ch(E) \in H^{2*}(X;\rationals)\\
    R(\hat{\ch}(E,\nabla)) &= P_{\ch}(C_1(\nabla^2),\dots) =
    \tr({-}\exp(\nabla^2/2\pi i)) \in \Omega^{2*}(X).
  \end{split}
\end{equation*}

\subsubsection{Differential Chern character transformation}

The second point of view of the differential Chern character is as a natural
transformation between differential K-theory and differential
cohomology. Indeed, we prove in \cite[Section 6]{bunke-2007} the following
theorem.

\begin{theorem}
  There is a unique natural transformation of differential cohomology theories
  \begin{equation*}
    \hat{\ch}\colon \hat K^*(\cdot)\to \hat H^*(\cdot; \rationals[u,u^{-1}])
  \end{equation*}
  with the following properties:
  \begin{enumerate}
  \item Compatibility with the Chern character in ordinary cohomology and with
    the action of forms,
    i.e.~the following diagram commutes:
    \begin{equation*}
      \begin{CD}
        \Omega(X)/\im(d) @>a>> \hat K(X) @>I>> K(X)\\
        @VV{=}V @VV{\hat{\ch}}V @VV{\ch}V\\
          \Omega(X)/\im(d) @>a>> \hat H(X;\rationals[u,u^{-1}]) @>I>>
          H(X:\rationals[u,u^{-1}]) \ .
      \end{CD}
    \end{equation*}
\item Compatibility with the curvature homomorphism, i.e.~the following
  diagram commutes
  \begin{equation*}
    \begin{CD}
      \hat K(X) @>R>> \Omega_{d=0}(X;\reals[u,u^{-1}])\\
      @VV{\hat{\ch}}V @VV=V\\
      \hat H(X;\rationals[u,u^{-1}]) @>R>> \Omega_{d=0}(X;\reals[u,u^{-1}]) \ .
    \end{CD}
  \end{equation*}
\item Compatibility with suspension, i.e.~the following diagram commutes
  \begin{equation*}
    \begin{CD}
      \hat K(S^1\times X) @>{\hat{\ch}}>> \hat H(S^1\times X;\rationals[u,u^{-1}])\\
      @VV{{\pr_2}_!}V @VV{{\pr_2}_!}V\\
      \hat K(X) @>{\hat{\ch}}>> \hat H(X;\rationals[u,u^{-1}]).
    \end{CD}
  \end{equation*}
  Here, we use the differential K-orientation of $\pr_2$ as in Section
  \ref{sec:suspension}. 
  \end{enumerate}

Moreover, if $x\in\hat K^0(X)$ is represented (as in \ref{sec:vect_with_conn}
or \ref{sec:geometric_families})  by the zero-dimensional family
$((E,\nabla),\rho)$ (with $\rho$ the additional form of odd degree), then
\begin{equation*}
  \hat{\ch}(x) = \hat{\ch}(E,\nabla)+a(\rho) \in \hat H^{2*}(X;{\rationals[u,u^{-1}]}).
\end{equation*}
In addition, $\hat{\ch}$ is a multiplicative natural transformation and becomes
an isomorphism after tensor product with $\rationals$.
\end{theorem}

\subsection{Differential Riemann-Roch theorem}

We are now in the situation to formulate the Riemann-Roch theorem, which
describes the relations between push-forward in differential K-theory and
differential cohomology and the differential Chern character. Let $p\colon
W\to B$ be a proper submersion with a differential K-orientation $o$ represented
by $(g^{T^vp},T^hp,\tilde \nabla,\sigma)$ as in Section \ref{sec:orientations}.

The Riemann Roch theorem asserts the commutativity of a diagram
\begin{equation*}
\begin{CD}
  \hat K(W) @>{\hat\ch}>> \hat H(W,\rationals)\\
  @VV{\hat p_!}V @VV{\hat p_!^A}V\\
  \hat K(B) @>{\hat\ch}>> \hat H(B,\rationals)\ .
\end{CD} 
\end{equation*}
Here $\hat p_!^A$ is the composition  of the cup product with a differential
rational cohomology class $\hat \hA^c(o)$ and the push-forward in differential
rational cohomology (uniquely determined by an ordinary orientation of $p$, in
particular by the K-orientation which is already fixed).

We have to define the refinement $\hat \hA(o)\in \hat H^{ev}(W,\rationals)$ of the
form $\hA^c(\tilde\nabla)\in \Omega^{ev}(W{;R[u,u^{-1}]})$.
The geometric data of $o$ determines a connection $\nabla^{T^vp}$ and hence a geometric bundle $\mathbf{T^vp}:=(T^vp,g^{T^vp},\nabla^{T^vp})$.  According to \cite{MR827262} we can define Pontryagin classes $$\hat p_{i}(\mathbf{T^vp})\in \hat H^{4i}(W,\integers)\ ,\quad i\ge 1\ .$$
The $\spinc$-structure gives rise to a Hermitean line bundle $L^2\to W$ with connection $ \nabla^{L^2}$ (see \eqref{eq:def_of_L2}).
We set $\bL^2:=(L^2,h^{L^2},\nabla^{L^2})$.
Again using \cite{MR827262},  we get a class
$$\hat c_1(\bL^2)\in \hat H^2(W,\integers).$$
% with curvature
% $R(\hat c_1(\bL^2))=2c_1(\tilde \nabla)$.

Inserting the  classes $u^{-2i}\hat p_{i}(\mathbf{T^vp})$ into that $\hA$-series $\hA(p_1,p_2,\dots)\in \rationals[[p_1,p_2\dots]]$ we  define
\begin{equation}\label{hahaha}\hat\hA(\mathbf{T^vp}):=\hA(\hat
  p_1(\mathbf{T^vp}),\hat p_2(\mathbf{T^vp}),\dots)\in \hat
  H^{{0}}(W,{\rationals[u,u^{-1}]})\ .\end{equation} 
\begin{definition}
We define
$$\hat \hA^c(o):=\hat\hA(\mathbf{T^vp})\wedge e^{\frac{1}{2u}\hat c_1(\bL^2)}
- a(\sigma)\in \hat H^{0}(W,\rationals[u,u^{-1}])\ .$$
\end{definition}
Note that $R(\hat \hA^c(o))=\hA^c(\tilde\nabla)$, with $\hA^c(\tilde\nabla)$
of \eqref{eq:A_hat_c}.

By \cite[Lemma 6.17]{bunke-2007}, $\hat\hA^c(o)$ indeed does not depend on the
particular representative of $o$. This follows from the homotopy formula.

We now define
\begin{equation*}
  \hat p_!^A\colon \hat H^*(W;\rationals[u,u^{-1}])\to \hat H^{*-n}(B;\rationals[u,u^{-1}]);\;x \mapsto \hat p_!(\hat\hA^c(o)\cup x)
\end{equation*}

The differential Riemann-Roch now reads
\begin{theorem}\label{RiemannRoch}
The following square commutes
\begin{equation*}
\begin{CD}
  \hat K(W) @>{\hat\ch}>> \hat H(W,\rationals[u,u^{-1}])\\
  @VV{\hat p_!}V @VV{\hat p_!^A}V\\
  \hat K(B) @>{\hat\ch}>> \hat H(B,\rationals[u,u^{-1}]).
\end{CD}
\end{equation*}
\end{theorem}

This diagram is compatible via the transformations $I$ with the maps of the
classical Riemann-Roch theorem.

\subsection{Differential Atiyah-Singer index theorem}
\label{sec:diff_AS}

We want to understand the Atiyah-Singer index theorem as the equality of the
analytic and the topological (family) index. To formulate a differential
version of this, the first step is to define the differential analytic and
topological (family) index.

We start with the proper submersion $p\colon W\to B$ with $n$-dimensional
fibers, and we think of $W\to B$ 
as the underlying family of manifolds on which the family of elliptic
operators shall be given. To be able to define a \emph{differential} index, we
have to choose geometry for $p\colon W\to B$, which amounts exactly to the
choice of data representing a differential orientation of $p$ as in Section
\ref{sec:orientations}, namely a tuple $(g^{T^vp},T^hp,\tilde \nabla,\sigma)$
consisting of vertical metric, horizontal distribution, fiberwise
$\spinc$-structure and compatible $\spinc$-connection. The form $\sigma$ could
of course be chosen equal to $0$.

We follow the definition of the analytic index as given by Freed and Lott in
\cite{freed10:_index_theor_in_differ_k_theor}. We start with a vector bundle
$(E,\nabla_E)$
with connection over $W$ (representing a class in $\hat K^0(W)$ using either
the model of Section \ref{sec:vect_with_conn} or
\ref{sec:geometric_families}). We then want to define the analytic index of
the family 
of $\spinc$-Dirac operators twisted by $(E,\nabla_E)$. This is based on
Bismut-Quillen
superconnections. Indeed, it uses the Bismut-Cheeger eta-form which mediates
between the Chern character of the finite dimensional index bundle (giving the
naive analytic index) and the Chern character of the Bismut superconnection.
{For details see below.}

The topological index of \cite{freed10:_index_theor_in_differ_k_theor} is
modelled closely after the classical definition of the topological index by
Atiyah and Singer. One factors $p\colon W\to B$ as a fiberwise embedding $W\to
S^N\times B$ and $\pr_2\colon S^N\times B\to B$. One then uses very explicit
formulas for the differential K-theory push-forward of the embedding $W\to
S^N\times B$, based on a model of differential K-theory using
currents. Finally, a K\"unneth decomposition of $\hat K(S^N\times B)$ gives an
explicit push-forward for $\pr_2\colon S^N\times B\to B$. The topological
index is defined as a modification of the composition of these two
push-forwards. It does not involve spectral analysis {nor does it require} the solution
of differential equations. It does use differential forms, so the term
``differential topological index'' is indeed quite appropriate.
{For details again see below.}

The main result of \cite{freed10:_index_theor_in_differ_k_theor} is then
\begin{theorem}\label{theo:diff_index_theorem}
  Differential analytic index and differential topological index both define homomorphism
  \begin{equation*}
   \ind^{an},\ \ind^{top}\colon  \hat K^0(W)\to \hat K^{-n}(B).
  \end{equation*}
Moreover, analytic and topological index coincide:
  \begin{equation*}
    \ind^{an}=\ind^{top}
  \end{equation*}
 Finally, $\ind^{an}=\ind^{top}$ fits into a commutative diagram
  \begin{equation*}
    \begin{CD}
      0 @>>> K\reals/\integers^{-1}(W) @>>> \hat K^0(W) @>R>>
      \Omega^0_{d=0}(W;\reals[u,u^{-1}])\\
      && @V{\ind^{an}=}V{\ind^{top}}V @V{\ind^{an}=}V{\ind^{top}}V @VV{ \omega\mapsto \int_{W/B} 
      \hA^c(\tilde\nabla)\wedge \omega% p_!^o
    }V \\
      0 @>>> K\reals/\integers^{-1-n}(B) @>>> \hat K^{-n}(B) @>R>> \Omega^{-n}_{d=0}(B;\reals[u,u^{-1}]).
    \end{CD}
  \end{equation*}
\end{theorem}

  In the following, we explain the ingredients of this formula.
  \begin{enumerate}
   \item On the differential form level, $\hA^c(\tilde\nabla)$ of
     \eqref{eq:A_hat_c} coincides with the form $Todd(\hat\nabla^W)$ of
     \cite[(2.14)]{freed10:_index_theor_in_differ_k_theor}.
   \item 
    For K-theory with coefficients in $\reals/\integers$, $\ind^{an}\colon
    K\reals/\integers^{-1}(W)\to 
    K\reals/\integers^{-1-n}(B)$ has been constructed in and is the main
    theme of \cite{MR1312690}. In particular, the equality of analytic and
    topological index in this context is proved there.
  \end{enumerate}
  
  \begin{proposition}
    In the end, of course, $\ind^{an}=\ind^{top}\colon \hat K^0(W)\to \hat
    K^{-n}(B)$ coincide with $\hat p_!\colon \hat K^0(W)\to \hat K^{-n}(B)$ of
    Section \ref{sec:integration_map}.
  \end{proposition}
The main point is that
    $\ind^{an}$ is defined as an honest analytic index, whereas $\hat p_!$
    only formally does so.

    In the following, we will assume that the fiber dimension $n$ is even. The
    case of odd $n$ is easily reduced to this via suspension-desuspension
    constructions using products with $S^1$.

\subsubsection{Analytic index in differential K-theory}
\label{sec:analytic_index}

    We now define the analytic index of a cycle {$(\cE,\phi)=(E,h_E,\nabla_E,\phi)$
   } for 
    $\hat K^0(W)$ given in terms of a vector bundle with connection and an
    auxiliary form $\phi$ as in Section \ref{sec:vect_with_conn}.  To get a
    cleaner picture, we assume that the family $D_E$ of twisted $\spinc$-Dirac
    operators over $B$, constructed as the geometric Dirac operators provided
    by the differential K-orientation data twisted with $(E,\nabla_E)$ has a
    kernel bundle $\ker(D_E)$ which comes with an induced Hermitean metric and
    compressed connection $\nabla_{\ker(D_E)}$. The Bismut-Cheeger eta-form in
    this situation is defined as
    \begin{equation*}
      \tilde \eta:= u^{-n/2} R_u \int_0^\infty
      STr\left(u^{-1}\frac{dA_s}{ds}e^{u^{-1}A_s^2}\right)\,ds\;\in\Omega^{-n-1}(B;\reals[u,u^{-1}])/\im(d). 
    \end{equation*}
    $R_u$ is introduced to simplify notation, it is induced from the ring
    homomorphism $\reals[u,u^{-1}]\to \reals[u,u^{-1}]; u\mapsto (2\pi i) u$.
    $A_s$ is the Bismut superconnection on the (typically infinite
    dimensional) bundle $\mathcal H$ over $B$ whose fiber over $b\in B$ is the
    space of sections of the twisted spinor bundle over $W_b=p^{-1}(B)$, the
    bundle on which $D_E$ acts. More precisely, for $s>0$
    \begin{equation*}
      A_s= s u^{1/2} D_E+\nabla^{\mathcal{H}}-s^{-1}u^{-1/2}c(R)/4.
    \end{equation*}
    Here, $\nabla^{\mathcal{H}}$ is a canonical connection on $\mathcal{H}$
    constructed out of the given connections, and $c(R)$ is Clifford
    multiplication by the curvature $2$-form of $\mathcal{H}$. For details
    about this construction, see \cite[Section 10]{MR1215720}. Note that,
    following Freed-Lott, powers of $s$ instead of powers of $s^{1/2}$ are
    used in the definition of the superconnection.

    The eta-form provides an interpolation between the Chern character form of
    the kernel bundle and of $E$ as follows (compare
    \cite[(3.11)]{freed10:_index_theor_in_differ_k_theor} and
    \cite[(0.6)]{bunke-2007}):
    \begin{equation}\label{eq:d_eta}
      d\tilde \eta = \int_{W/B}\hat A^c(\tilde\nabla)\wedge\ch(\nabla_E)
      - \ch(\nabla^{\ker(D_E)}).
    \end{equation}

 Following \cite[Definition 3.12]{freed10:_index_theor_in_differ_k_theor}, one
 now defines for {$(\cE,\phi)=(E,h_E,\nabla_E,\phi)$}
 \begin{equation}
   \label{eq:def_an_ind}
   \ind^{an}(\cE,\phi) :=
   \left((\ker(D_E),h_{\ker(D_E)},\nabla^{ker(D_E)}),\int_{W/B}\hat 
   A^c(\tilde\nabla)\wedge\phi+\tilde \eta\right),
\end{equation}
given by the zero dimensional family $\ker(D_E)$ with its Hermitean metric and
connection, and the differential form part which uses the eta form.

In \cite[Theorem 6.2]{freed10:_index_theor_in_differ_k_theor} it is shown that
this formula indeed factors through a map $\ind^{an}$ on differential K-theory
as stated in Theorem \ref{theo:diff_index_theorem}.
{Alternatively one could use \cite[Corollary 5.5]{bunke-2007} which 
immediately implies that
$\ind^{an}(\cE,\phi)$ is equal to the push-forward $\hat p_{!}^{o}(\cE,\phi)$
as defined in (\ref{eq300}).
}

 The construction of the analytic differential
index in the general case, i.e.~if the kernels do not form a bundle, is
carried out by a perturbation to reduce to the special case treated so far.

\subsubsection{Topological index in differential K-theory}

The topological index is defined in two steps. One chooses a \emph{fiberwise
  isometric} fiberwise embedding $i\colon W\to S^N\times B$ of even
codimension, where the target is equipped with the product structure. We
assume that we have a
differential $\spinc$-structure on the normal bundle $\nu$ of the embedding
which is compatible with structures on $W$ and on $S^N\times B$ (in the sense
of \cite[Section 5]{freed10:_index_theor_in_differ_k_theor}).

Given a
Hermitean bundle with connection $(E,\nabla_E)$ and a 
differential form $\phi$ on $B$, one now has to construct the differential
push-forward $i_*$ of {$(\cE,\phi)=(E,h_E,\nabla_E,\phi)$}. This is based on the Thom 
homomorphism in differential K-theory. More precisely, one has to choose a
($\integers/2\integers$-graded) Hermitean bundle $F$ on $S^N\times B$  with
connection together 
with an endomorphism $V$ such that $V$ is invertible outside $i(W)$ but the
kernel of $V$ is isomorphic (as geometric bundle) to the tensor product of $E$
with the spinor bundle, and a suitably defined compression of the covariant
derivative of $V$ under this isomorphism becomes Clifford multiplication. In
\cite[Definition 1.3]{BismutZhang} in this situation a transgression form
$\gamma$ is defined which interpolated between the Chern form of $F$ and the
image image of $\ch(\nabla_E)$ under the differential form $\spinc$-Thom
homomorphism. 

We then define
\begin{equation}
  \label{eq:embedding_push_forward}
   \hat{i}_!(\cE{,\phi}):= \left((F,h_F,\nabla_F),\phi\wedge \hat
     A^o(\nabla_\nu)^{-1}\wedge \delta_X -\gamma- C'\right) \in \hat
     K^{N-n}(S^N\times B) 
\end{equation}
  where $\hat A^o(\nabla_\nu)^{-1}\wedge \delta_X$ is the current representing
  the $\spinc$-Thom form and $C'$ is a further correction term which vanishes
  if the horizontal distribution of $W\to B$ is the restriction of the product
  horizontal distribution of $S^n\times B$ under the embedding $i$. 

We now define
\begin{equation}\label{eq:top_ind}
\ind^{top}(\cE{,\phi}) := \hat p_!(\hat i_!(\cE,{\phi})) \in \hat K^{-n}(B),
\end{equation}
where $p\colon S^N\times B\to B$ is the projection and we use the product
structure to define the $\spinc$-orientation.

Finally, we observe that for the product $S^N\times B$ we can use the K\"unneth
type formula to obtain an explicit formula for $\hat p_!( x)$. More precisely,
use the K\"unneth formula in ordinary K-theory to write the underlying
K-theory class $I(X)=p^*a+ \pr_1^*({t})\cdot p^*b$ where ${t}$ is a second
additive generator (besides $1$) of $K^*(S^N)$. We lift all the classes
$a,b,t$ to classes $A,B,T$ in differential K-theory and then write $
X= p^*A+pr_1^*(T)\cdot p^* B + a(\alpha)$ with a suitable differential form
$\alpha$. One then obtains by
\cite[(5.31)]{freed10:_index_theor_in_differ_k_theor} 
\begin{equation}
  \label{eq:product_push}
  \hat p_!(X) = B + a(\int_{S^N\times B/B} \hat A^o\wedge \alpha)
\end{equation}
where $\hat A^o{\in \Omega^{0}(S^{N}\times B;\reals[u,u^{-1}])}$ is given by the product structure on $p\colon S^N\times B\to
B$. By \cite[Corollary 7.36]{freed10:_index_theor_in_differ_k_theor} this
construction indeed factors through a homomorphism $\ind^{top}$ on
differential K-theory. 

\subsubsection{Proof of the differential index theorem}
\label{sec:proof_of_differential_ind_theorem}

It remains to prove that for any ${(\cE,\phi)}=(E,h_E,\nabla_E,\phi)$ as above,
$\ind^{top}(\cE)-\ind^{an}(\cE)=0$. The constructions are carried out in such
a way that the desired identity holds for the images under $R$. Therefore
$\ind^{top}(\cE)-\ind^{an}(\cE)=:T \in K^{-n-1}(B;\reals/\integers)$ is a flat
differential K-theory class. Now
Freed and Lott in \cite{freed10:_index_theor_in_differ_k_theor} follow the
method of proof of the $\reals/\integers$-index theorem of
\cite{MR1312690}. One has to show that the pairing of $T$ with $K_{-n-1}(B)$
(the ordinary K-homology of $B$) vanishes. These pairings are given by reduced
$\eta$-invariants (provided we have a kernel bundle). In the end, one has to
establish certain identities for $\eta$-invariants on $W$, on $B$ and on
$S^N\times B$. These follow from adiabatic limit considerations and ---to deal
with the embedding $i\colon W\to S^N\times B$ also the main result
\cite[Theorem 2.2]{BismutZhang}. 

For the general case, one has to follow the effect of the perturbations which
reduce to the case of a kernel bundle. 

\subsubsection{Related work}

The thesis \cite{klonoff08:_index_theor_in_k_theor} discusses a model of
even differential K-theory using vector bundles with connection and
push-forward maps in this model. The work culminates in the special case of
the index theorem for differential K-theory if $B$ is the point.

\section{Twisted differential K-theory}
\label{sec:twisted_K}

Usual cohomology theories often have severe limitations when dealing with
situations in which orientations are required, but not present.
This happens in particular when one wants to describe the cohomological
properties of a fiber bundle which is not oriented for the cohomology theory
one wants to study. Closely related is the non-existence of integration maps
for non-orientable bundles or more generally non-orientable maps.

This problem is solved by using cohomology with twisted coefficients. For
ordinary cohomology, this is just described by a local coefficient system,
which one can easily implement e.g.~in a \v{C}ech description of cohomology,
compare e.g.~\cite[Chapter 10]{MR658304}. Twists have been
 introduced for generalized cohomology theories and successfully
 used. {We refer to \cite{MR2271789} for the approach to twisted cohomology via parameterized spectra and to \cite{2008arXiv0810.4535A} for a construction using infinity categories.
}
 
 {
In particular, twisted K-theory has been studied extensively, motivated by 
by the classification of D-brane charges in the presence of a
background B-field
as  discussed in Section
\ref{sec:diff_cohom_and_physcis}.}

\subsection{Twists for ordinary K-theory}
\label{sec:twisted_ordinary_K_theory}

The most general twists for a multiplicative generalized cohomology theory {represented by an $E_{\infty}$-ring spectrum $E$}
are (up to 
equivalence) described by degree $0$ cohomology classes with coefficients in
${bgl}_1(E)$, 
where $gl_1(E)$ is the spectrum of units of $E$ {and
  $bgl_1(E):=gl_{1}(E)[1]$ is its one-fold deloop.} In the case of K-theory, {the spectrum $bgl_{1}(K)$} contains the summand
$H\integers[3]$. In other words, there is a subgroup of the isomorphism
classes of twists for K-theory on a space $B$ given by
$H^3(B;\integers)$. Most authors concentrate on these twists.

However, it is inappropriate to think only in terms of isomorphism classes of
twists. 
 The twists always form a pointed groupoid (with a trivial object). {Technically, one can take the path groupoid of the mapping space 
 \begin{equation}\label{mapp}\mathrm{map}_{Spectra}(\Sigma^{\infty} X_{+},bgl_{1}(E))\ .\end{equation}
The twisted cohomology 
theory is more appropriately understood as a functor from this groupoid to graded abelian groups.
This path groupoid is the truncation of the $\infty$-groupoid given by the mapping space itself which should really be considered as right object. In the present paper we prefer the truncation since it is used in most applications and the generalization to the differential case. }

In
particular, a given twist usually will have non-trivial automorphisms, and these
automorphisms act non-trivially on the twisted cohomology. In our case,
the automorphisms of the trivial K-theory twist on $B$ are given by
$H^2(B;\integers)$. Because of the non-trivial automorphisms, for an
isomorphism class of twists, e.g.~$c\in H^3(B;\integers)$ it does in general
not make sense to talk about "`the"' twisted K-theory group $K^{c}(B)$. Only
the isomorphism class of this group is well defined, but this is not sufficient
e.g.~if one wants to discuss functorial properties.

{
As indicated above one usually does not work with the most general
kind of twists determined homotopy theoretically by $gl_{1}(E)$ but with a more explicit class closely tied to the 
relevant geometric situation. We therefore collect, following
\cite[Section 2, Section 3.1]{freed-2007}, the standard properties of a twisted extension of the cohomology theory $E$ in an axiomatic manner.}

 \begin{definition}\label{def:twisted_extension}
  Let $E$ be a generalized cohomology theory. An extension of $E$ to
  cohomology with twists consists of the following data:
  \begin{itemize}
  \item For every space $X$ of a (pointed) groupoid $\Twist_X$.
  \item For every continuous map $f\colon Y\to X$ a functor
    $f^*\colon\Twist_X\to \Twist_Y$ which is (weakly) functorial in $f$. Even
    more: the association $X\to \Twist_X$ should
    become a weak presheaf of groupoids. {In most examples this presheaf of groupoids  satisfies descend for open coverings and
    therefore forms  a stack in topological spaces.}
  \item Define then $\Twist$ {as the Grothendieck construction of the presheaf above,} i.e.~the category with objects $(X,\tau)$ where $X$ is
    a space and $\tau\in\Twist_X$ and morphisms from $(X,\tau_X)$ to $(Y,\tau_Y)$
    consisting of a map $f\colon X\to Y$ together with an isomorphism
    $\tau_X\to f^*\tau_Y$. Define the category $\Twist^2$ of \emph{pairs in twists} with
    objects $(X,A,\tau)$ as before, but where $A\subset X$ and $\tau$ is a
    twist on $X$.
  \item The twisted version of $E$ is then a contravariant functor from
    $\Twist^2$ to graded 
    abelian groups,
    \begin{equation*}
      (X,A,\tau)\mapsto E^{\tau+n}(X,A)
    \end{equation*}
   together with natural transformation
   \begin{equation*}
     \delta\colon E^{\tau+n+1}(X,A)\to E^{\tau+n}(A,\emptyset).
   \end{equation*}
   These have to satisfy the following properties:
    \item homotopy invariance: here, a homotopy between $f,g\colon
      (X,\tau_X)\to (Y,\tau_Y)$ is a morphism $h\colon (X\times
      [0,1],\pr^*\tau_X)\to (Y,\tau_Y)$ such that $i_0\circ h=f$ and $i_1\circ
      h=g$. The morphism of pairs $i_k\colon (X,\tau_X)\to (X\times[0,1],\pr^*\tau_X)$
      uses the identity morphism o twists $i_k^*\pr^*\tau_X \to \tau_X$.
    \item long exact sequence of the pair:
      \begin{equation*}
       \to E^{\tau+n}(X,A)\to E^{\tau+n}(X)\to
       E^{\tau+n}(A)\xrightarrow{\delta} E^{\tau+n-1}(X,A)\to
      \end{equation*}

    \item {excision} isomorphism
    \item wedge axiom: if $(X,A,\tau)=\disjointunion_{i\in I}
      (X_i,A_i,\tau_i)$ then the natural map 
      \begin{equation*}
  E^{\tau+n}(X,A)\to \prod_{i\in I} E^{\tau+n}(X_i,A_i)
    \end{equation*}
    is an isomorphism.
  \item For the base point $0\in\Twist_X$, we require that
    $E^{0+n}(X,A)=E^n(X,A)$ with the given definition of $E^n$.
  \end{itemize}

Often, one will require additional structure, in particular a monoidal
structure on $\Twist_X$ which one typically writes additively. Then one
requires a natural bilinear product
\begin{equation*}
  E^{\tau_1+n}(X,A)\tensor E^{\tau_2+m}(X,A)\to E^{\tau_1+\tau_2+n+m}(X,A)
\end{equation*}
which should be associative and graded commutative up to the natural
isomorphism of twistings coming from the monoidal structure.

In this situation, one would also require a functorial and compatible
push-forward for a proper map between smooth manifolds $f\colon X\to Y$
\begin{equation*}
  f_!\colon E^{f^*\tau+o(f)+*}(X)\to E^{\tau+*-(\dim X-\dim Y)}(Y)
\end{equation*}
where $o(f)$ is an orientation twist associated to the map $f$. An
$E$-orientation of the map $f$ will give rise to a trivialization $o(f)\to 0$,
\textcolor{red}{so} that for an oriented map one has a push-forward
\begin{equation*}
  f_!\colon E^{f^*\tau+*}(X)\to E^{\tau+*-(\dim X-\dim Y)}(Y).
\end{equation*}
\end{definition}

As usual with cohomology theories, there are variants, depending on which
category of spaces and pairs of spaces one considers, and for which situations
precisely one requires excision.

{In the approach   of \cite{2008arXiv0810.4535A}
these axioms can easily be realized.
} 
\begin{example}\label{ex:twisted-de-Rham}
  As mentioned in Section \ref{sec:diff_cohom_and_physcis}, one can twist
  de Rham cohomology, defined on the category of smooth manifolds (possibly with
  boundary), as follows. Let $N$ be a graded commutative algebra, e.g.~$N=\reals[u,u^{-1}]$.
  \begin{itemize}
   \item $\Twist_X:=\Omega_{d=0}^1(X;N)$, the closed $N$-valued forms of total
     degree $1$, with pullback the
    usual pullback. The base point is the form $\Omega=0$.
  \item The morphisms from $\Omega_1$ to $\Omega_2$ are given by all 
    forms $\eta\in \Omega^0(X;N)$ with $\Omega_2=\Omega_1+d\eta$. Composition is
    defined as the sum of differential forms.
  \item $H_{dR}^{\Omega+n+ev}(X,A):=
    \ker(d^\Omega|_{\bigoplus_{k\in\integers}\Omega^{n+2k}})/\im(d^\Omega)$,
    with differential $d^\Omega(\omega):=d\omega+\Omega$. Note that the fact
    that $\Omega$ is closed implies that this is indeed a
    differential.
  \item Given a morphism $\eta\in\Omega^{{0}}(X;N)$ from $\Omega_1$ to
    $\Omega_1+d\eta$, define the induced isomorphism of twisted de Rham groups
    \begin{equation*}
      \eta^*\colon H_{dR}^{\Omega_1+n}(X{;N})\to H^{\Omega_1+d\eta+n}_{dR}(X{;N});
      [\omega]\mapsto [\omega\cup\exp(-\eta)]
    \end{equation*}
  \item The sum of forms defines a strictly symmetric monoidal structure
    on $\Twist$ and the cup product of differential forms induces an
    associative and graded commutative product structure on twisted de Rham cohomology.
  \end{itemize}
\end{example}

\begin{remark}\label{rem:small_twisted_de_Rham}
  A variant of Example \ref{ex:twisted-de-Rham} uses as twists only the
  differential forms $\Twist'_X:=\Omega_{d=0}^3(X)\cdot u \subset
  \Omega^1_{d=0}(X;N)$. This is particularly relevant for the comparison with
  K-theory.

  Note that the for the isomorphism classes of twists one obtains
  $\pi_0(\Twist'_X)= H^3_{dR}(X)$.
\end{remark}

In the case of K-theory, there are many different models for the groupoid of
twists we consider. A
particularly simple version is the following: Let $U$ be the
unitary group of an infinite dimensional separable Hilbert space $H$ (with norm
topology) and
$PU:=U/S^1$ where $S^1$ is the center, the multiples of the identity. Because
of Kuiper's theorem, $U$ is contractible and $PU$ has the homotopy type of the
Eilenberg-MacLane space $K(\integers,2)$. Let $K$ be the $C^*$-algebra of
compact operators on $H$. Conjugation defines an action of $PU$ on $K$ by
$C^*$-algebra automorphism. 

\begin{example}
  Assume that $B$ is compact. $\Twist_B$, the groupoid of twists for K-theory on $B$ is
  now defined as the category of 
  principal $PU$-bundles over $B$, with morphisms the homotopy classes of
  bundle isomorphisms. If $\tau$ is such a bundle, we can form the associated
  bundle of $C^*$-algebras $\tau\times_{PU}K$. The sections of this bundle
  form themselves a $C^*$-algebra. Now set
  $K^{\tau+*}(B):=K^*(\Gamma(\tau\times_{PU}K))$. A {isomorphism} $\beta\colon
  \tau'\to \tau$ of $PU$-principal bundles induces an isomorphism of
  associated K-bundles and therefore also of the $C^*$-algebras of sections,
  which finally induces a (functorial) isomorphism $\beta^*\colon
  K^{*+\tau'}(B)\to K^{*+\tau}(B)$.
\end{example}

{
Algebras of sections of K-bundles like
  $\tau\times_{BU}K$ are called "`continuous trace algebras"' and are an
  important object of study in operator algebras, compare
  e.g.~\cite{MR1634408}.}
This point of view of twisted K-theory ---using continuous trace algebras---
is exploited and {developed} e.g.~in \cite{MR2116734,MR920145,Tu-Xu,Tu-Xu-LG}.

 Homotopy invariance of K-theory of
  $C^*$-algebras implies that $\beta^*$ depends only on the homotopy class of
  $\beta$. The trivial twist is the trivial bundle $PU\times B$. Its
  automorphisms are given by maps from $B$ to $PU$. As $PU=K(\integers,2)$,
  the set of
  homotopy classes of such maps is $H^2(B;\integers)$. For the special case of
  torsion classes in $H^3(B;\integers)$, this model has first been considered
  in \cite{DonovanKaroubi}. More precisely, {this paper} uses bundles of
  finite dimensional matrix algebras over $B$ instead of K-bundles which is
  exactly the reason why only torsion twists occur. The general case is
  studied in \cite{Tu-Xu-LG, Tu-Xu}. Another, closely related model for the
  twists is given by $U(1)$-bundle gerbes. This point of view is studied
  e.g.~in \cite{BCMMS}. 

  Note that $PU$ also acts by conjugation on the space $\Fred$ of Fredholm
  operators on $H$. The latter is a model for the zeroth space of the K-theory
  spectrum. Given a twist $\tau$, we can form the associated bundle
  $\tau\times_{PU}\Fred$. We can then define $K^{0+\tau}(B)$ alternatively as
  the homotopy classes of sections of $\tau\times_{PU}\Fred$. This model is
  used e.g.~in \cite{Atiyah-Segal1, Atiyah-Segal2}. One can define
  $K^{1+\tau}(B)$ by using an appropriate classifying space for $K^1$ instead
  of $\Fred$ which is a classifying space for $K^0$.

  Obviously, a more refined version of this construction uses bundle of
  spectra (also called parameterized spectra) instead of bundles of spaces. A
  very precise version of such a model, with a satisfactory description of a
  product structure, of orientation and of the natural transformation from
  twisted $\spinc$-cobordism to twisted K-theory corresponding to the Atiyah
  orientation has been worked out in \cite{Waldmueller}. When dealing with
  bundles, it is necessary to deal with objects and maps on the nose, and not
  only up to homotopy.

  Our description suggests a further "`categorification"' of the concept of
  (twisted)  generalized cohomology theory. In the same way as twists have to
  be considered as a groupoid, one should also think 
  of (twisted) generalized cohomology  as a groupoid. The objects of this
  groupoid are the
  cocycles, and a cochain $c$ of shifted degree (modulo boundaries) is a
  morphism from $x$ to $x'=x+dc$. {This would require a  two-groupoid of twists, e.g.~a two-truncation of the mapping space
  (\ref{mapp}).}
  
{ More explicitly, in
in our example} we might think of $H^3(X;\integers)$ as the groupoid whose
  objects are isomorphism classes of principal $PU$-bundles over $X$,
  morphisms are $PU$-bundle maps and $2$-morphism are homotopies of
  $PU$-bundle maps. Similarly, we might think of $K^{0+\tau}(B)$ as the
  groupoid whose objects are sections of $\tau\times_{PU}\Fred$ and with
  morphisms homotopies of sections. 
{
If one likes $\infty$-categories then one could  consider twisted cohomology as an $\infty$-functor which associates to a twist
 $\tau\in \mathrm{map}_{Spectra}(\Sigma^{\infty} X_{+},gl_{1}(E))$
 the spectrum of sections of the associated bundle $E_{\tau}$ of spectra.
This can be made precise using \cite{2008arXiv0810.4535A}.  In this picture it is easy to  implement additional structures like multiplication or  push-forward.
 }

{
In the truncated groupoid picture most of this has been carried out in the even more
  elaborate equivariant situation in \cite[Section 3]{freed-2007}. }The model there is
  based on the construction of twisted K-theory spectra.

\subsection{Twisted differential K-theory}
\label{sec:twisted_diff_K:special}

  To define the concept of a twisted differential generalized cohomology
  theory, one has to combine the concept of twist with the concept of
  differential extension (which is \emph{not} a cohomology theory, but there
  the deviation is well under control). One does need groupoids of
  \emph{differential twists} which contain differential form
  information. Along the way, one will need an appropriate Chern character to
  twisted de Rham cohomology.

  The following definition, what in general a twisted differential cohomology
  theory should be, follows essentially \cite[Appendix A.3]{kahle-2009}.

  \begin{definition}\label{def:twisted_diff_general}
    A differential extension of a twisted cohomology theory as in Definition
    \ref{def:twisted_extension} consists of the
    following data:
    \begin{itemize}
    \item for each smooth manifold $X$ a groupoid $\Twist_{\hat E,X}$,
      together with (weakly) functorial pullback along smooth maps.
      They form a weak presheaf of groupoids. As above, we can ``combine'' all
      these groupoids to the category $\Twist_{\hat E}$.
    \item natural functors
      \begin{equation*}
        \begin{split}
          F\colon \Twist_{\hat E,X} &\to \Twist_X\\
          \Curv\colon \Twist_{\hat E,X} &\to \Omega^1_{d=0}(X;N).
        \end{split}
      \end{equation*}
      Here and in the following, we lift the action of $\Omega^0(X;N)$ on
      twisted de Rham cohomology with the same formula to the twisted de Rham complexes.
    \item To each $\tau\in\Twist_{\hat E,X}$ we assign a Chern character
      \begin{equation*}
        \ch^\tau\colon E^{F(\tau)+*}(X) \to H_{dR}^{\Curv(\tau)+*}(X;N),
      \end{equation*}
      natural with respect to pullback.
    \item The differential twisted extension of $E$ is then a functor from
     the category of differential twists $\Twist_{\hat E}$ to graded abelian groups:
     \begin{equation*}
       (X,\tau) \mapsto \hat E^{\tau+*}(X).
     \end{equation*}
     Note that this includes functoriality with respect to pullback along
     smooth maps and along isomorphism of twists.
   \item There are natural (for pullback along smooth maps and along
     isomorphism of twists) transformations
    \begin{equation*}
       \begin{split}
         I\colon &\hat E^{\tau+*}(X)\to E^{F(\tau)+*}(X)\\
         R\colon &\hat E^{\tau+*}(X)\to \Omega_{d^{\Curv(\tau)}=0}^*(X;N)\\
         a\colon &\Omega^{*-1}(X;N)/\im(d^{\Curv(\tau)}) \to \hat E^{\tau+*}(X).
       \end{split}
     \end{equation*}
   \item They satisfy 
     \begin{equation*}
R\circ a=d^{\Curv(\tau)}\qquad\text{and}\qquad \ch^\tau\circ I =
     \pr\circ R
   \end{equation*}
where $\pr\colon \Omega_{d^{\Curv(\tau)}=0}^*(X;N)\to
     H^{\Curv(\tau)}_{dR}(X{;N})$ is the canonical projection.
   \item Using these, we get exact sequences
     \begin{equation*}
  {E^{F(\tau)+*-1}(X) } \xrightarrow{}    \Omega^{*-1}(X{;N})/\im(d^{\Curv(\tau)})\xrightarrow{a} \hat
         E^{\tau+*}(X)\xrightarrow{I} E^{F(\tau)+*}(X)\to 0
     \end{equation*}
    \end{itemize}
    Additionally, one would typically like to have a compatible product
    structure, as in Definition \ref{def:twisted_extension}, with an adopted
    rule for the compatibility of the transformation $a$.

    Finally, if one has a product structure, one would like to have a
    push-forward along smooth maps (or at least proper submersions) $f\colon
    X\to Y$ of the form
    \begin{equation*}
      f_!\colon \hat E^{f^*\tau+o(f)+*}(X)\to \hat E^{\tau+*+(\dim X-\dim Y)}(Y),
    \end{equation*}
    where $o(f)$ is a differential orientation twist associated to $f$. A
    differential $E$-orientation should induce a trivialization $o(f)\to 0$ so
    that in this case one gets a push-forward
    \begin{equation*}
      f_!\colon \hat E^{f^*\tau+*}(X)\to \hat E^{\tau+*+(\dim X-\dim Y)}(Y).
    \end{equation*}
  \end{definition}

A first attempt toward a definition and description of twisted differential
K-theory is given in \cite{MR2518992}, although not exactly in the setting of
Definition \ref{def:twisted_extension}. The main problems are of course:
\begin{enumerate}
\item construction of the groupoid of differential twists
\item construction of the differential cohomology groups
\item construction of the push-forward.
\end{enumerate}

\cite{MR2518992} works with $U(1)$-banded bundle gerbes with connection and
curving as objects of the groupoid of twists, and the curvature $3$-form of
this connection and curving is the transformation $\Curv$ (on objects). Given
such a twist, they construct a principal $PU(H)$-bundle. Their twisted
differential K-theory is then based on sections of associated bundles of
Fredholm operators and explicitly constructed locally defined vector bundles
with connection. For the rather elaborate precise definition, we refer to
\cite[Section 3]{MR2518992}.

\begin{definition}
  The twists for $X$ used in \cite[Definition A.1]{kahle-2009} are \emph{geometric
    central extensions}. Such a geometric central extension is
  \begin{enumerate}
\item 
    a groupoid $(P_0,P_1)$ with a local equivalence to the trivial groupoid
    $(X,X)$,
  \item a central $U(1)$-extension of groupoids $L\to P_1$
  \item in
    particular, $L\to P_1$ is a $U(1)$-principal bundle, and another part of
    the data is
    a connection $\nabla$ on this principal bundle,
  \item moreover, $L\to P_1$ being a central extension means one has over
    $P_2=P_1\times_{P_0}P_1$ an isomorphism of line bundles $\lambda\colon
    \pr_1^*L\tensor \pr_2^*L\to \circ^*L$ using the two projections and the
    composition of arrows $\pr_1,\pr_2,\circ\colon P_2\to P_1$. $\lambda$
    should satisfy the cocycle condition, i.e.~the two different ways to map
    $L_h\tensor L_g\tensor L_f$ to $L_{h\circ g\circ f}$ on
    $P_3=P_1\times_{P_0}P_1\times_{P_0}P_1$ coincide.
  \item a $2$-form $\omega\in \Omega^2(P_0)$.
  \end{enumerate}
  These ingredients have to satisfy certain compatibility conditions explained
  in \cite[Definition A.1]{kahle-2009}. In particular,
  $p_1^*\omega-p_0^*\omega = \frac{\sqrt{-1}}{2\pi}\Omega^\nabla$,
 and $\lambda$ is an isomorphism of line bundles with connection.
\end{definition}

In \cite{kahle-2009} it is observed that no construction of twisted
differential K-theory with their twists (the geometric central extensions) is
available yet, but one certainly expects that such a construction is
possible. 

\subsection{T-duality}
\label{sec:top_T-dual}

Motivated from string theory, T-duality is expected to be an equivalence of
low energy limits of type IIA/B superstring theories on T-dual pairs. In
particular, as D-brane charges are classified by twisted K-theory, T-duality
predicts a canonical isomorphism between appropriate twisted K-theory groups
of the underlying topological spaces of the T-dual pairs. This prediction has
been made mathematically rigorous under the term ``topological T-duality''. It
is investigated e.g.~in \cite{MR2080959,MR2130624,MR2287642,MR2482327}. 

We briefly introduce into the mathematical setup as proposed in \cite[Section
2]{MR2287642}, compare also \cite[Section 4]{MR2482327}. 

\begin{definition}\label{def:T-duality-triple}
  We let $T^n=U(1)^n$ be the $n$-dimensional real torus, considered as Lie
  group. Let $B$ be a topological space (often with some restrictions, e.g.~to
  be a compact CW-complex).
 
  A T-duality triple consists of two $T^n$-principal bundles $E,E'$ over the
  common base space $B$, and twists $\tau\in\Twist_E$, $\tau'\in\Twist_{E'}$
  for K-theory. The third ingredient of a T-duality triple is an isomorphism
  of twists $u\colon p^*\tau\to (p')^*\tau'$ over the \emph{correspondence
    space} $E\times_BE'$ with the two canonical projections $p\colon
  E\times_BE'\to E$ and $p'\colon E\times_BE'\to {E}'$.

The twists and the isomorphism $u$ of twists have to satisfy certain
conditions. These are most transparent if $n=1$. In this case, they simply say
that
\begin{equation*}
  \int_{E/B} [\tau]=c_1(E');\qquad\int_{E'/B}[\tau']=c_1(E),
\end{equation*}
where $[\tau]\in H^3(E;\integers)$ is the characteristic class determined and
determining the isomorphism class of the twist $\tau$. Moreover, restricted to
a point each $x\in B$, $\tau|_{E_x}$ and $\tau'_{E'_x}$ are canonically
trivialized (because $E_x\iso U(1)\iso E'_x$ have vanishing $H^2$ and
$H^3$). Consequently, using the induced trivializations, the restriction of
$u$ to the fiber over $x$ becomes an automorphism of the trivial twist and
therefore is classified by an element in $H^2(E_x\times_xE'_x)\iso
H^2(U(1)\times U(1);\integers)\cong \integers$. We require that this element is the canonical
generator.

For the details for general $n$, we refer to \cite[Section 2]{bunke-2008},
where this is again treated using cohomology, or \cite[Defintion
4.1.3]{MR2482327} where the language of stacks is used.

 Of course, in this setting we first have to choose appropriate
  data for a   twisted extension of K-theory, e.g.~the model where twists are
  $PU$-principal bundles or $U(1)$-banded gerbes.
\end{definition}

\begin{definition}
  Given a T-duality triple $((E,\tau),(E',\tau'),u)$ as in Definition
  \ref{def:T-duality-triple}, we define the T-duality transformation of
  twisted K-theory
  \begin{equation}\label{eq:T-duality-transform}
   T:=p'_!u^*p^*\colon K^{*+\tau}(E)\to K^{*-n+\tau'}(E').
  \end{equation}
  It is defined as the composition of pull-back to the correspondence space,
  using $u$ to map $\tau$-twisted K-theory to $\tau'$-twisted K-theory and
  finally integration along $p'$, where we use the fact that $T^n$-principal
  bundles are canonically oriented for any cohomology theory, in particular
  for K-theory.
\end{definition}

The main results of \cite{MR2287642} and \cite{MR2482327} concern
\begin{enumerate}
\item the classification of T-duality triples: there is e.g.~a universal
  T-duality 
  triple over a classifying space $\mathbf{R}_n$ of such triples whose
  homotopy type is 
  computed: it is the homotopy fiber in the sequence
  \begin{equation}\label{eq:R_n}
  \mathbf{R}_N\to   K(\integers^n,2)\times K(\integers^n,2)\xrightarrow{\cup} K(\integers,4)
  \end{equation}
where the map $\cup$ is the composition of the map associated to the usual
cup-product and the standard 
scalar product of the coefficient group $\integers^n\otimes
\integers^n\to\integers$. 

Note, therefore, that up to equivalence all the information of a T-duality
triple is given by two $T^n$-principal bundles $P,P'$ with Chern classes
$c_1,\dots, c_n$, $c_1',\dots,c_n'$ with $\sum c_ic_i'=0$, together with an
explicit trivialization of the cycle representing this product (e.g.~a lift of
its classifying map to the homotopy fiber $\mathbf{R}_n$). 

In \cite{MR2287642}, we also discuss, which pairs $(E,\tau)$ can be part of a
T-duality
  triple and in how many ways. For $n=1$, this is always the case, even in a
  unique way (up to equivalence). For $n>1$, both these assertions are wrong in
  general. 
\item The T-duality transform $T$ of \eqref{eq:T-duality-transform} is always
  an isomorphism (compare \cite[Theorem 6.2]{MR2287642}).
\end{enumerate}

\subsection{T-duality and differential K-theory}
\label{sec:diff_T_duality}

Alex Kahle and Alessandro Valentino \cite{kahle-2009} study the effect of
T-duality in differential K-theory. They follow the approach of Section
\ref{sec:top_T-dual}, i.e.~they first make a precise definition of a
differential T-duality triple \cite[Definition 2.1]{kahle-2009} (there called
``pair''). 

However, they make very efficient use of the higher structures of
differential cohomology alluded to above.

\begin{definition}
  Fix a base space $B$. 
  A differential T-duality triple on $B$ according to Kahle and Valentino
  \cite[Definition 2.1]{kahle-2009} first consists of two objects
  $\mathcal{P}=(P,\nabla),\mathcal{P}'=(P',\nabla')$ in the groupoid of cycles
  for differential cohomology with coefficients in $\integers^n$, given by two
  $T^n$-bundles with connection over $B$, $\pi\colon P\to B$, $\pi'\colon
  P'\to B$.

  Secondly, using a product on the level of the groupoid of cycles, form
  $\mathcal{P}\cdot \mathcal{P}'$, a cycle for $\hat H^4(B;\integers)$ (on the
  level of the coefficients $\integers^n$, for this multiplication we use the 
  standard inner product). The   last ingredient for a T-duality triple 
  is an isomorphism in the groupoid of cycles for differential $H^4$:
  \begin{equation*}
    \sigma\colon 0\to \mathcal{P}\cdot \mathcal{P}'.
  \end{equation*}
  Note that the existence of such a trivialization is a strong condition on
  $\mathcal{P},\mathcal{P}'$. 
\end{definition}

Observe that this description is very much in line with the description of the
homotopy type of the classifying space for topological T-duality triples
\eqref{eq:R_n} and the resulting description of T-duality triples.

To obtain the twists for differential K-theory one expects for a differential
T-duality triple, Kahle and Valentino argue as follows:

The pullback of $\mathcal{P}$ to the total space $P$ of the underlying bundle
has a canonical trivialization, and similarly for $\mathcal{P'}$. This
trivialization can be multiplied with $\pi^*\mathcal{P}'$ to give a
trivialization of $\pi^*\mathcal{P}\cdot \pi^*\mathcal{P}'$. The composition of
$\pi^*\sigma$ with the inverse of this is an automorphism of the trivial
object $0$ and therefore defines a cycle $\hat\tau$ for the third differential
cohomology of $P$. Similarly, we obtain a cycle $\hat\tau'$ for the third
differential cohomology of $P'$. Finally, in \cite[Lemma 2.2]{kahle-2009}, it
is shown how the canonical trivializations of $\pi^*\mathcal{P}$ and
$(\pi')^*\mathcal{P}'$ give rise to a morphism $\hat u$ in the groupoid of
cycles for 
$\hat H^3(P\times_BP')$ from $p^*\hat\tau$ to $(p')^*\hat\tau'$, where $p\colon
P\times_BP'\to P$, $p'\colon P\times_BP'\to P'$ are the canonical
projections.

The crucial points assumed by Kahle-Valentino is
\begin{enumerate}
\item to have a groupoid cycle model for differential cohomology where cycles
  for $\hat H^2$ are principal $U(1)$-bundles with connection and where one
  has a multiplication with good properties on the level of cycles
\item to have a model for a twisted extension of differential K-theory, where
  the groupoid of cycles for $\hat H^3$ is exactly the groupoid of twists.
\end{enumerate}
Let us repeat that, at the moment, no complete construction of twisted
differential K-theory satisfying these requirements seems to be available.

With these assumptions, it is now immediate how to define a T-duality
transformation in twisted differential K-theory (assuming that ``integration
along $T^n$-principal bundles with connection'' is also defined for the
twisted differential K-theory at hand):
\begin{equation}\label{eq:diff_T_transfo}
  \hat T:= \hat p'_!\circ \hat u^*\circ p^*\colon \hat K^{*+\hat\tau}(P)\to \hat
  K^{*-n+\hat\tau'}(P'). 
\end{equation}
Here $\hat u^*$ is the isomorphism induced by the isomorphism of differential
twists $\hat u$ of \cite[Lemma 2.2]{kahle-2009} as above.

However, there is one observation to be made: upon application of the
curvature transformation, simple calculations show that $\hat T$ can never be
surjective, as the forms in the image of its differential form analog have a
very specific invariance property under the action of $T^n$.  

\begin{definition}
  For a $T^n$-principal bundle like $P$, let $\hat\tau$ be a cycle for
  differential cohomology and twist for differential K-theory as above and
  assume that the differential form $R(\hat\tau)$ is $T^n$-invariant. Define the
  \emph{geometrically invariant subgroup} of twisted differential K-theory as
  \begin{equation*}
    \hat K^{*+\hat\tau}(P)^{T^n}:=\{x\in \hat K^{*+\hat\tau}(P)\mid
    g^*R(x)=R(x)\;\forall g\in T^n\}.
  \end{equation*}
\end{definition}

With this notion, Kahle and Valentino prove their main result \cite[Theorem
2.4]{kahle-2009}.

\begin{theorem}
  The differential T-duality transform $\hat T$ of \eqref{eq:diff_T_transfo}
  preserves the geometrically invariant subgroups and defines an isomorphism
  \begin{equation*}
    \hat T\colon\hat K^{*+\hat\tau}(P)^{T^n}\to \hat K^{*-n+\hat\tau'}(P')
  \end{equation*}
\end{theorem}

The main point of the proof of this theorem is the construction of the
transformation in such a way that it is compatible with all the
transformations given for differential K-theory. One then has to check/use
that 
the transformation is an isomorphism for geometrically invariant forms (the
image under the curvature homomorphism of geometrically invariant differential
K-theory) and for topological twisted K-theory. The proof then concludes using
the five lemma.

\section{Applications of differential K-theory}

Differential K-theory is a natural home for many well known, and hopefully
some new, typically secondary invariants. In this section, we want to present
some examples of this kind. To be able to do this, we start with a couple of
elementary calculations.

\begin{lemma}
  \begin{equation*}
  {  \hat K^1(*)=\reals/\integers;\qquad\hat
    K^0(*)=\integers,\qquad\hat{K}^1_{flat}(*)=\reals/\integers;\qquad\hat K^0_{flat}(*)=\{0\}}
  \end{equation*}
  as follows directly from the short exact sequence \eqref{exax}. 
\end{lemma}

\subsection{Holonomy}

Let $(V,\nabla)$ be a Hermitian vector bundle of rank $n$ over $S^1$ with
unitary 
connection and with holonomy $\phi$ (well defined modulo conjugation in
$U(n)$). $(V,\nabla)$ defines a geometric family $\mathcal{V}$ and therefore
an element in differential K-theory $[\mathcal{V},0]$. By \cite[Lemma
5.3]{bunke-2007} 
\begin{lemma}
  \begin{equation*}
    [\mathcal{V},0] = a(\frac{1}{2\pi i} \det(\phi)).
  \end{equation*}
\end{lemma}

\subsection{$\integers/k\integers$-invariants}

Recall that a $\integers/k\integers$-manifold is a manifold $W$ with boundary
together with a manifold $X$ together with a diffeomorphism $f\colon \boundary
W\to 
\underbrace{X\disjointunion\dots\disjointunion X}_{\text{$n$ copies}}$.

We now associate to a  family of $\integers/k\integers$-manifolds over $B$ a
class in $\hat K_{flat}(B)$. 
\begin{definition} A \emph{geometric family of
    $\integers/k\integers$-manifolds} is a 
triple $(\cW,\cE,\phi)$, where 
$\cW$ is a geometric family with boundary, $\cE$ is a geometric family without
boundary, and 
$\phi\colon \partial\cW\stackrel{\sim}{\to} k\cE$ is an isomorphism of the
boundary of $\cW$ with $k$ copies of $\cE$. 
\end{definition} 
 We define $u(\cW,\cE,\phi):=[\cE,-\frac{1}{k}\Omega(\cW)]\in\hat K(B)$.
\begin{lemma}
We have
$u(\cW,\cE,\phi)\in \hat K_{flat}(B)$. This class
is a $k$-torsion class.
It only depends on the underlying differential-topological data.
\end{lemma}

\begin{theorem}
Let $B=*$ and $\dim(\cW)$ be even.
Then $u(\cW,\cE,\phi)\in \hat K_{flat}^1(*)\cong \reals/\integers$.
Let $i_k\colon \integers/k\integers\rightarrow \reals/\integers$ the embedding which sends $1+k\integers$ to
$\frac{1}{k}$.
Then
$$i_k(\ind_a(\bar W))=u(\cW,\cE,\phi)\ , $$
where  $i_k(\ind_a(\bar W))\in\integers/k\integers$ is the index of the
$\integers/k\integers$-manifold $\bar W$, and where we use the notation of
\cite{MR1144425}. 
\end{theorem}

\subsection{Reduced eta-invariants}

Let $\pi$ be a finite group.
We construct a transformation
$$ \phi\colon \Omega_d^{\spinc}(BU(n)\times  B\pi)\rightarrow \hat
K^{-d}_{flat}(*)\ .$$ 
Let $f\colon M\rightarrow BU(n)\times B\pi$ represent $[M,f]\in  \Omega^{Spin^c}(BU(n)\times B
\pi)$. This map determines a $\pi$-covering $p\colon \tilde M\rightarrow M$
and an 
$n$-dimensional complex vector bundle $V\rightarrow M$.
We choose a Riemannian metric $g^{TM}$ and a $\spinc$-connection
$\tilde\nabla$.   
These structures determine a differential $K$-orientation of $t\colon M\to
*$. We further fix a metric $h^V$ and a connection $\nabla^V$ in order to
define a geometric bundle $\bV:=(V,h^V,\nabla^V)$ and the associated geometric
family $\cV$. 
The pull back of $g^{TM}$ and $\tilde \nabla$ via $\tilde M\to M$ fixes a 
differential $K$-orientation of $\tilde t\colon \tilde M\to *$. 

Then we set
$$\phi([M,f]):=[\tilde t_!(p^*\cV)\sqcup_* |\pi|t_!\cV^{op},0]\in \hat
K_{flat}(*)\ .$$ 
In \cite[Section 5.10]{bunke-2007} it is shown that this class only depends on
the bordism class of $[M,f]$.

More generally, without even the assumption that $\pi$ is finite, choose two
finite dimensional representations $\rho_1,\rho_2$ of $\pi$ of the same
dimension, with associated flat bundles $F_1,F_2$. Replace in the above
$\abs{\pi}t_!\cV$ by $t_!(\cV\tensor F_2)$ and $\tilde t_!(p^*\cV)$ by
$t_!(\cV\tensor F_1)$.

Note that this boils down to the previous case if $\rho_1$ is the regular
representation of the finite group $\pi$, and $\rho_2$ is
$\complexs^{\abs{\pi}}$, where $\complexs$ stands for the trivial
representation. 

\begin{proposition}
  This construction defines a homomorphism
  \begin{equation*}
    \phi_{\rho_1,\rho_2}\colon \Omega_d^{\spinc}(BU(n)\times B\pi)\to \hat
    K_{flat}^{-d}(*).
  \end{equation*}
  If $d$ is even, the target group is trivial. If $d$ is odd, $\hat
  K_{flat}^{-d}(*)\cong \reals/\integers$. In this case, $\phi_{\rho_1,\rho_2}$
  coincides with the reduced rho-invariant of
  Atiyah-Patodi-Singer. 
\end{proposition}

The construction immediately generalizes to a parameterized version: to a
smooth family of $d$-dimensional $\spinc$-manifolds parameterized by $B$, with
a family of $\complexs^n$-vector bundles 
and also of $\pi$-coverings one associates in the same way a class in
$\hat{K}^{-d}_{flat}(B)$.

For details of all of this, compare \cite[Section 5.10]{bunke-2007}.

\subsection{$e$-invariant}

A framed manifold is a manifold $M$ together with a trivialization of its
tangent bundle. 

\cite[Proposition 5.22]{bunke-2007} states that a bundle of framed
$n$-manifolds $\pi\colon E\to B$ has a canonical differential K-orientation,
given by 
the fiberwise $\spinc$-structure which comes from the trivialization, and the
$\spinc$-connection which again comes from the trivial connection (form part
$0$). We then define
\begin{equation*}
  e([\pi\colon E\to B]):= \hat\pi_!(1) \in \hat K^{-n}_{flat}(B).
\end{equation*}
The push-down is with respect to the canonical $\hat K$-orientation of $\pi$,
and the flatness of the connection of this differential K-orientation in the
end implies that $R(\hat\pi_1(1))=0$.

\begin{proposition}
  If $B=*$ and $n$ is odd, 
  $e([B])\in K^{-1}_{flat}(*)=\reals/\integers$.

  This class coincides with the Adam's classical $e$-invariant for the stable
  homotopy groups, identified with the framed bordism groups.
\end{proposition}

For details, compare \cite[Section 5.11]{bunke-2007}.

\subsection{Secondary invariants for string bordism}
\label{sec:string_bordism}

In \cite{2009arXiv0912.4875B}, using spectral invariants of Dirac operators,
Bunke and Naumann construct a secondary Witten genus, a bordism invariant of
string manifolds. They use differential cohomology to facilitate some of their
calculations, compare e.g.~\cite[Lemma 2.2]{2009arXiv0912.4875B}.

\section{% Miscellaneous matters}
% \label{sec:miscellania}
% \subsection{
  Equivariant differential K-theory and orbifold differential K-theory}
\label{sec:equivariant_and_orbifold}

As explained in Section \ref{sec:diff_cohom_and_physcis}, one of the
motivations for the study  of differential K-theory comes from physics, where
fields in abelian gauge theories are suggested to be modelled by cocycles for
differential K-theory and where some of the main features are captured by the
properties of differential cohomology theories. In Section \ref{sec:diff_T_duality}
we have seen how this is successfully applied to T-duality, another important
subject motivated by string theory.

In particular for the latter, however, mathematical physics also requires the
study of singular spaces. Such singular spaces often arise as quotients of
smooth manifolds by the action of a group, which is one reason why equivariant
situations are important. We would therefore like to study  differential
K-theory for singular spaces and equivariant differential
K-theory. Unfortunately, these theories are not well understood yet. 

In
\cite{bunke-orbifold2009}, we construct differential $K$-theory of
representable smooth orbifolds, i.e.~global quotients of a manifold by a
compact group. The construction is based on equivariant local index theory in
the spirit of Section \ref{sec:geometric_families}. The relevant Chern
character takes values in delocalized de Rham cohomology of the orbifold. In
case of a global quotient by a finite group, this is defined in terms of the de
Rham complexes of the fixed point sets. We obtain
a ring valued functor with the usual properties of a differential extension of a
cohomology theory. For proper submersions (with smooth fibers) we construct a
push-forward map in differential  orbifold $K$-theory. Finally, we construct a
non-degenerate intersection pairing with values in $\complexs/\integers$ for
the subclass of 
smooth orbifolds which can be written as global quotients by a finite group
action.  We construct a real subfunctor of our theory, where the
pairing restricts to a non-degenerate $\reals/\integers$-valued
pairing. Indeed, we use 
in that paper the language of (differentiable \'etale) stacks which turns out
to be particularly convenient. 

In \cite[Section 5.4]{Szabo-Valentino}, %\cite[Section
                                %6]{valentino08:_k_theor_d_branes_and},
 a model for equivariant differential K-theory in the spirit of Section
 \ref{sec:def_via_classifying_maps} is constructed. It uses the fact that
 there are very nice models for the classifying space for equivariant
 K-theory. As target for the Chern character on equivariant K-theory
 \cite{Szabo-Valentino} uses Bredon cohomology with coefficients in the
 representation ring tensored with $\reals$. Via a de Rham isomorphism, this
 is canonically isomorphic to delocalized de Rham cohomology. As in the
 non-equivariant case, this model is not so well suited to the construction of
 a product structure and of push-forwards, which are therefore not discussed
 in \cite{Szabo-Valentino}. However, in \cite[Section 6]{Szabo-Valentino} is
 described how equivariant differential K-theory can be used to described
 Ramond-Ramond fields and their flux quantization condition in orbifolds of
 type II superstring theory. 

The preprint \cite{ortiz-2009} gives yet another construction of differential 
equivariant $K$-theory for finite group actions along the lines of
\cite{MR2192936}, i.e.~of Section \ref{sec:def_via_classifying_maps}. 
Moreover, it constructs a product and push-forward to a
point. The  
constructions are mainly homotopy theoretical.%  At the time of writing,
% however, it seems that not all of the arguments in \cite{ortiz-2009} can be
% validated. \textbf{CHECK TO WHICH EXTEND THIS IS THE CASE and add the details,
%   if appropriate.}
 Ortiz in \cite{ortiz-2009} raises the interesting question
\cite[Conjecture 6.1]{ortiz-2009} of identifying his push-forward in
analytic terms. In the model of \cite{bunke-orbifold2009}, in view of the
geometric construction of the push-forward and the analytic nature of the
relations,  the conjectured relation 
 is essentially a tautology. See \cite[Corollary 5.5]{bunke-2007} for a more
 general statement in the non-equivariant case.  \cite[Conjecture
 6.1]{ortiz-2009}
 would be an immediate consequence of a theorem stating that any two models of
 equivariant differential $K$-theory for finite group actions are canonically
 isomorphic in a way  compatible with integration.  
 It seems to be plausible that the method of \cite{BSuniqueness} extends to the
 equivariant case. In \cite{ortiz-2009}, the conjectured equivariant index
 formula is proved in a number of special cases, e.g.~if $\Gamma$ is trivial,
 or in case the $G$-manifold is a $G$- boundary. 

\bigskip
\noindent
\textbf{Acknowledgements}. Thomas Schick was partially funded by the Courant
Research Center “Higher order structures in 
Mathematics” within the German initiative of excellence. 

% \subsection{Cohomology operations}

% Gomi encodes certain secondary cohomology operations via the product in
% differential cohomology in \cite{MR2509714}. 

%\nocite{*}

{\small
\bibliographystyle{plain}
\bibliography{Differential_K}
}

\end{document}